\documentclass{amsart}
\usepackage{amssymb,amsmath}

\usepackage{float}
\usepackage{subfig,epsfig}
\usepackage{graphicx}
\usepackage{enumerate}
\usepackage{color}
\usepackage{hyperref}

\makeatletter

\newcommand{\Rmnum}[1]{\expandafter\@slowromancap\romannumeral #1@}
\makeatother

\newtheorem{theorem}{Theorem}[section]
\newtheorem{lemma}[theorem]{Lemma}
\newtheorem{proposition}[theorem]{Proposition}

\theoremstyle{definition}
\newtheorem{definition}[theorem]{Definition}

\theoremstyle{remark}
\newtheorem{remark}[theorem]{Remark}

\numberwithin{equation}{section}

\begin{document}

\title{Levi-type Schur-Sergeev duality for general linear super groups}

    %Information for first author
\author{Di Wang}
%%    Address of record for the research reported here
\address{School of mathematical sciences, East China Normal University,
Shanghai, 200241,  People's Republic of China}
\email{52195500007@stu.ecnu.edu.cn}
\thanks{This work is supported by
the National Natural Science Foundation of China (NSFC) Gr:12071136}
\thanks{The author thanks the supervisor Bin Shu for his suggestion for this research and discussion,
also expresses her thanks to the referees for helpful comments.}
% Information for second author
%\author{Jianbo Fang}
%\address{
%School of Mathematical Sciences,
%Tongji University,
%Shanghai, 200092, People's Republic of China}
%\address{School of Mathematics and Statistics,
%Chuxiong Normal University,
%Chuxiong, 675000, People's Republic of China}
%\email{fjbwcj@126.com}
%\thanks{The second author is the corresponding author.}
%and supported by the National Science
%Foundation of China (No.11671298) and
%the Science Research Project of Shanghai (No.16ZR1439200).}

%    General info
%\subjclass[2010]{52A40; 53A04; 53C44}

%\date{}



\begin{abstract}
 %This paper investigates a kind of polynomial representations of a supergroup %$GL(m|n)\times\mathbf{G}_m$, which is $r-th$tensor of enhanced vector spaces, and proves the %Schur superalgebra corresponding to
 %$GL(m|n)\times\mathbf{G}_m$ forms duality with a class of algebras named weak degenerate %double Hecke algebras,denoted by $\underline{\mathcal{H}}_{r}$, the definition of which %comes from \cite{B-Y-Y2020} with a little change.
% Thanks for the help of Professor Bin Shu.
In this note,  we investigate a kind of double centralizer property for general linear supergroups. For the super space $V=\mathbb{K}^{m\mid n}$ over an algebraically closed field $\mathbb{K}$ whose characteristic is not equal to $2$, we consider its $\mathbb{Z}_2$-homogeneous one-dimensional extension $\underline V=V\oplus\mathbb{K}v$, and the natural action of the supergroup
$\tilde G:=\text{GL}(V)\times \textbf{G}_m$ on $\underline V$. Then we have  the tensor product supermodule ($\underline{V}^{\otimes r}$, $\rho_r$) of $\tilde G$. We present a kind of generalized Schur-Sergeev duality which is said that  the Schur superalgebras $S'(m|n,r)$ of $\tilde G$ and a so-called weak degenerate double Hecke algebra $\underline{\mathcal{H}}_r$ are double centralizers.
The weak degenerate double Hecke algebra is an infinite dimensional algebra, which has a natural representation on the tensor product space.
This notion comes from \cite{B-Y-Y2020}, with a little modification.
\end{abstract}
\maketitle

\section{Introduction}\label{sec0}
In the super mathematics, there are some results similar to the classical Schur-Weyl duality introduced in \cite{G1980} and \cite{G-W2009}.

Let $\mathbb{K}$ be a field with characteristic $p\neq 2$, and $V:=\mathbb{K}^{m|n}$ the super vector space over $\mathbb{K}$.
Denote by  $\mathrm{GL}(m|n)$ the supergroup functor from the category of super commutative $\mathbb{K}$-superalgebras to the category of groups,
which is called the general linear supergroup. More details about supergroups can be found in \cite{C-C-F2011} and \cite{B-K2001}.
Let $\mathfrak{S}_d$ be the symmetric group of $d$ elements.
$S(m|n)$ is the Schur superalgebra produced from the action of $\mathrm{GL(m|n)}$ on $V^{\otimes d}$. The super-version of Schur-Weyl duality over $\mathbb{K}$ is introduced in \cite[section 5]{B-K2002}:
\begin{align*}
	\mathrm{End}_{\mathbb{K}\mathfrak{S}_d}(V^{\otimes d})\cong &S(m|n);\\
	&\mathrm{End}_{S(m|n)}(V^{\otimes d})\cong \mathbb{K}\mathfrak{S}_d, d\leq m+n.
\end{align*}
And for Lie superalgebras, there are double centralizer properties for type A Lie superalgebras and queer Lie superalgebras over complex numbers with more details in \cite{C-W2012}, which are called Schur-Sergeev dualities.

 Denote by $W$ the $n$ dimensional vector space over complex number $\mathbb{C}$, and $\underline{W}$ the enhanced vector space which is the one-dimensional extension of $W$. The image $\mathfrak{L}(n,r)$ of the Levi subgroup $\mathrm{GL(W)}\times \mathbf{G}_m$ in $\mathrm{End}(\underline{W}^{\otimes r})$ is called Levi Schur algebra in \cite{B-Y-Y2020}.
In \cite{B-Y-Y2020}, the  authors investigated the structure and the duality of $\mathfrak{L}(n,r)$ and construct an algebraic model so-called degenerate double Hecke algebras. Then \cite{B-Y-Y2020} studied the degenerate double Hecke algebra of $\mathfrak{S}_r$, denoted by $\mathcal{H}_r$, and its representation $\Xi$ on $\underline{W}^{\otimes r}$.
Then the following duality called Levi Schur-Weyl duality was proved in \cite{B-Y-Y2020}:
\begin{align*}
\mathrm{End}_{\mathfrak{L}(n,r)}(\underline{W}^{\otimes r})=&\Xi(\mathcal{H}_r);\\
&\mathrm{End}_{\Xi(\mathcal{H}_r)}(\underline{W}^{\otimes r})=\mathfrak{L}(n,r).
\end{align*}

The purpose of the present  paper  is to establish the double centralizer property of $\mathrm{GL(m|n)}\times \mathbf{G}_m$ in $\mathrm{End}_{\mathbb{K}}(\underline{V}^{\otimes r})$.
We still use the algebraic model produced in \cite{B-Y-Y2020}.
We extend the super vector space $V$ by an $\mathbb{Z}_2$-homogeneous vector $v$, and also say that the new super vector space is the enhanced super vector space of $V$, denoted by $\underline{V}$.
If $\bar{v}=\bar{1}$, the representation of $\mathcal{H}_r$ doesn't satisfy the equation (3.5) $\mathbf{s}_i\mathbf{x}_\sigma=\mathbf{x}_\sigma=\mathbf{x}_\sigma\mathbf{s}_i$ in \cite{B-Y-Y2020}.
So we change it to $\mathbf{s}_i\mathbf{x}_\sigma=\mathbf{x}_\sigma\mathbf{s}_i$, and name the new algebra by
weak degenerate double Hecke algebra, denoted it by $\mathcal{\underline{H}}_r$. We still use $\Xi$ to denote the representation of $\mathcal{\underline{H}}_r$ on $\underline{V}^{\otimes r}$.

This paper can be divided into three sections. In the first section, we introduce the Schur superalgebras $S'(m|n,r)$, and other basic concepts, notations and conclusions we will use in this paper.
In the second part, we introduce the weak degenerate double Hecke algebras and its representation on $\underline{V}^{\otimes r}$. Denote by $D(m|n,r)$ the image of its representation on $\underline{V}^{\otimes r}$.  In the third section, we prove the Levi-type Schur-Sergeev duality for general linear super groups:
\begin{align*}
	S'(m|n,r)\cong&\mathrm{End}_{D(m|n,r)}(\underline{V}^{\otimes r});\\
	&\mathrm{End}_{S'(m|n,r)}(\underline{V}^{\otimes r})\cong D(m|n,r) \text{ for } r\leq m+n.
\end{align*}

\section{Preliminaries}\label{sec1}
Let $\mathbb{K}$ be an algebraically closed field with characteristic $p\neq 2$.
\subsection{The general linear supergroup}
The notions and notations in this section are the same as in \cite[section 2]{B-K2002}.
\begin{definition}$($\cite[section 2]{B-K2002}$)$
The general linear supergroup $\mathrm{GL(m|n)}$ is a supergroup functor from the category of super commutative superalgebras to the category of groups defined on a super commutative superalgebra $A$ by letting $\mathrm{GL(m|n)}(A)$ be the group of all invertible $(m+n)\times(m+n)$ matrices of the form

\begin{equation}\label{equ:4.1}
	\begin{pmatrix}
		W  &X \\
		Y  &Z
	\end{pmatrix}
\end{equation}
where $W$ is an $m\times m$ matrix with entries in $A_{\bar{0}}$,
$X$ is an $m\times n$ matrix with entries in $A_{\bar{1}}$,
$Y$ is an $n\times m$ matrix with entries in $A_{\bar{1}}$,
$Z$ is an $n\times n$ matrix with entries in $A_{\bar{0}}$.
\end{definition}

Denote by $T^W_{ij}$ (respectively $T^X_{ij}$, $T^Y_{ij}$, $T^Z_{ij}$) the function mapping $W$ (respectively $X$, $Y$, $Z$)to its $(i,j)$entry.
Then the  coordinate ring of $\mathrm{GL(m|n)}$ can be expressed by $\mathbb{K}[\mathrm{GL(m|n)}]=\mathbb{K}[T^W_{ij},1\leq i,j\leq m;T^X_{ij},1\leq i\leq m,1\leq j\leq n;T^Y_{ij},1\leq i\leq n,1\leq j\leq m;T^Z_{ij},1\leq i,j\leq n]_{\mathrm{det}}$.
Let the parities of all $T^W_{ij}$ and $T^Z_{ij}$ be even, and the parities of all $T^X_{ij}$ and $T^Y_{ij}$ be odd.

Define $\bar{i}=\bar{0}$ for $i=1,...,m$, and $\bar{i}=\bar{1}$ for $i=m+1,...,m+n$.
Let $T_{ij}$ be the function mapping a matrix to its $(i,j)$entry,
and the parity of $T_{ij}$ be $\bar{i}+\bar{j}$.
Then
$$T_{ij}=
\begin{cases}
	&T^W_{ij}, \text{ if } 1\leq i,j\leq m \\
	&T^X_{i,j-m},\text{ if } 1\leq i\leq m \text{ and } m+1\leq j\leq m+n\\
	&T^Y_{i-m,j}, \text{ if } m+1\leq i\leq m+n \text{ and } 1\leq j\leq m\\
	&T^Z_{i-m,j-m}, \text{ if }  m+1\leq i,j\leq m+n
\end{cases}.$$

Then $\mathbb{K}[\mathrm{GL(m|n)}]$ is a Hopf superalgebra with the comultiplication and the counit
$$\Delta(T_{ij})=\sum_{h=1}^{m+n}T_{ih}\otimes T_{hj},$$
$$\varepsilon(T_{ij})=\delta_{ij}$$
for all $1\leq i,j\leq m+n$.

Let $\tilde{T}_{ij}=(-1)^{\bar{i}(\bar{i}+\bar{j})}T_{ij}$ for $1\leq i,j\leq m+n$.
\subsection{The supergroup $\mathrm{GL(m\mid n)}\times \mathbf{G_{m}}$}
\begin{definition}
	
	Let $\mathrm{\tilde G}:=\mathrm{GL(m|n)}\times \mathbf{G_{m}}$ be a functor from the category of super commutative superalgebras to the category of groups defined on a super commutative superalgebra $A$ by letting $\mathrm{\tilde G}(A)$ be the group of all invertible $(m+1+n)\times(m+1+n)$ matrices of the form
	\begin{equation}\label{equ:4.1}
		\begin{pmatrix}
			W & 0 &X \\
			0&a  &0 \\
			Y & 0 &Z
		\end{pmatrix}
	\end{equation}
	where $W$, $X$, $Y$ and $Z$ have the same form as above,
	and $a$ is an invertible elements in $A_{\bar{0}}.$
	
	The elements in $\mathrm{\tilde G}(A)$ can be written as $(g,a)$ where $g\in \mathrm{GL(m|n)}(A)$ and $a$  is an invertible element in $A_{\bar{0}}$.
\end{definition}
Let $$S_{ij}=
\begin{cases}
	&T^W_{ij}, \text{ if } 1\leq i,j\leq m \\
&T^X_{i,j-m},\text{ if } 1\leq i\leq m \text{ and } m+1\leq j\leq m+n\\
&T^Y_{i-m,j}, \text{ if } m+1\leq i\leq m+n \text{ and } 1\leq j\leq m\\
&T^Z_{i-m,j-m}, \text{ if }  m+1\leq i,j\leq m+n
\end{cases}$$
and $S:=T_{m+1,m+1}$, the function mapping a matrix to its $(m+1,m+1)$ entry.
Recall that $\bar{i}=\bar{0}$ for $i=1,...,m$, and $\bar{i}=\bar{1}$ for $i=m+1,...,m+n$.
We let $\tilde S_{ij}:=(-1)^{\bar{i}(\bar{i}+\bar{j})}S_{ij}$ for $1\leq i,j\leq m+n$.
The parity of $\tilde{S}_{ij}$ is still $\bar{i}+\bar{j}$
and $\bar{S}=\bar{0}$.

Then the coordinate algebra $\mathbb{K}[\mathrm{\tilde G}]$ is a super subalgebra of $\mathbb{K}[\mathrm{GL(m+1|n)}]$ and $\mathbb{K}[\mathrm{GL(m|n+1)}]$,
with $\mathbb{K}[\mathrm{\tilde G}]=\mathbb{K}[\tilde{S}_{ij},S,1\leq i,j\leq m+n]_{\mathrm{det}}$.
The coordinate algebra $\mathbb{K}[\mathrm{\tilde G}]$ is naturally a Hopf superalgebra with coproduct, counit defining as below

$$\Delta(\tilde{S}_{ij})=\sum_{h=1}^{m+n}(-1)^{(\bar{i}+\bar{h})(\bar{h}+\bar{j})}\tilde{S}_{ih}\otimes \tilde{S}_{hj},1\leq i,j\leq m+n,$$
$$\Delta(S)=S\otimes S,$$
$$\varepsilon(\tilde{S}_{ij})=(-1)^{\bar{i}(\bar{i}+\bar{j})}\delta_{ij},1\leq i,j\leq m+n,$$
$$\varepsilon(S)=1.$$

\subsection{Notions and notations}
Keep the same notions and notations as those in \cite[section 2]{B-K2002}.

For a super coalgebra $C$ with the comultiplication $\Delta$ and its dual superalgebra $C^*$,
the operator $\overline{\otimes}$ means $(x\overline{\otimes}y)(f\otimes g)=(-1)^{\bar{y}\bar{f}}x(f)y(g)$
for homogeneous elements $x,y\in C^*$, $f,g\in C$.

For a positive integer  $l$,
denote by $I(m|n,l)$ the set of all functions from $\{1,...,l\}$ to $\{1,...,m+n\}$. For $\mathbf{i}\in I(m|n,l)$, write $\mathbf{i}=(i_1,...,i_l)$ with $i_j\in \{1,...,m+n\},1\leq j\leq l$. Define $\mathbf{\bar{i}}=\bar{i}_1+..+\bar{i}_l$.

Let $\epsilon_{\mathbf{i}}=(\bar{i}_1,...,\bar{i}_l)\in \mathbb{Z}^{l}_{2}$ for $\mathbf{i}\in I(m|n,l)$. For $\epsilon=(\epsilon_{1},...,\epsilon_{l})\in \mathbb{Z}^{l}_{2}$, $\delta=(\delta_1,..,\delta_l)\in \mathbb{Z}^{l}_{2}$, and
$w\in \mathfrak{S}_l$, the symmetric group,
let
$$ \alpha(\epsilon,\delta)=\prod_{1\leq s<t\leq l}(-1)^{\epsilon_t\delta s},$$
$$ \gamma(\epsilon,w)=\prod_{1\leq s<t\leq l,\atop w^{-1}s>w^{-1}t}(-1)^{\epsilon_s\epsilon_t}.$$
Define $\mathbf{i}\cdot w=(i_1,...,i_l)\cdot w=(i_{w1},...,i_{wl})$ for $w\in \mathfrak{S}_l$ and $\mathbf{i}\in I(m|n,l)$.

The following equation is used frequently in the present paper. It can be showed by a direct computation.
\begin{align*}
	&\alpha(\epsilon_{\mathbf{i}w}+\epsilon_{\mathbf{j}w},\epsilon_{\mathbf{j}w})
	=\prod_{1\leq s<t\leq l}(-1)^{(\bar{i}_{wt}+\bar{j}_{wt})(\bar{j}_{ws})}\\
	&=\prod_{1\leq w^{-1}s<w^{-1}t\leq l}(-1)^{(\bar{i}_{t}+\bar{j}_{t})(\bar{j}_{s})}\\
	%&=\prod_{1\leq w^{-1}s<w^{-1}t\leq l,\atop1\leq s<t\leq l}(-1)^{(\bar{i}_{t}+\bar{j}_{t})(\bar{j}_{s})}\prod_{1\leq w^{-1}s<w^{-1}t\leq l,1\leq t<s\leq l}(-1)^{(\bar{i}_{t}+\bar{j}_{t})(\bar{j}_{s})}\\
	%&=\prod_{1\leq w^{-1}s<w^{-1}t\leq l,\atop1\leq s<t\leq l}(-1)^{(\bar{i}_{t}+\bar{j}_{t})(\bar{j}_{s})}\prod_{1\leq w^{-1}t<w^{-1}s\leq l,1\leq s<t\leq l}(-1)^{(\bar{i}_{s}+\bar{j}_{s})(\bar{j}_{t})}\\
	&=\prod_{1\leq w^{-1}s<w^{-1}t\leq l,\atop1\leq s<t\leq l}(-1)^{(\bar{i}_{t}+\bar{j}_{t})(\bar{j}_{s})}\prod_{1\leq w^{-1}t<w^{-1}s\leq l,\atop1\leq s<t\leq l}(-1)^{(\bar{i}_{s}+\bar{j}_{s})(\bar{j}_{t})}\\
	&~\prod_{1\leq w^{-1}t<w^{-1}s\leq l,\atop1\leq s<t\leq l}(-1)^{(\bar{i}_{t}+\bar{j}_{t})(\bar{j}_{s})}\prod_{1\leq w^{-1}t<w^{-1}s\leq l,\atop1\leq s<t\leq l}(-1)^{(\bar{i}_{t}+\bar{j}_{t})(\bar{j}_{s})}\\
	&=\prod_{1\leq s<t\leq l}(-1)^{(\bar{i}_{t}+\bar{j}_{t})(\bar{j}_{s})}\prod_{1\leq w^{-1}t<w^{-1}s\leq l,\atop1\leq s<t\leq l}(-1)^{(\bar{i}_{s}+\bar{j}_{s})(\bar{j}_{t})+(\bar{i}_{t}+\bar{j}_{t})(\bar{j}_{s})}\\
	%&=\prod_{1\leq s<t\leq r}(-1)^{(\bar{i}_{t}+\bar{j}_{t})(\bar{j}_{s})}\prod_{1\leq w^{-1}t<w^{-1}s\leq r,1\leq s<t\leq r}(-1)^{\bar{i}_{s}\bar{j}_{t}+\bar{i}_{t}\bar{j}_{s}}\\
	%&=\prod_{1\leq s<t\leq l}(-1)^{(\bar{i}_{t}+\bar{j}_{t})(\bar{j}_{s})}\prod_{1\leq w^{-1}t<w^{-1}s\leq l,\atop1\leq s<t\leq l}(-1)^{(\bar{i}_{s}+\bar{j}_s)(\bar{i}_t+\bar{j}_{t})+\bar{i}_{s}\bar{i}_{t}+\bar{j}_{s}\bar{j}_{t}}\\
	 &=\alpha(\epsilon_{\mathbf{i}}+\epsilon_{\mathbf{j}},\epsilon_{\mathbf{j}})\gamma(\epsilon_{\mathbf{i}}+\epsilon_{\mathbf{j}},w)\gamma(\epsilon_{\mathbf{i}},w)\gamma(\epsilon_{\mathbf{j}},w).
\end{align*}

Write $(\mathbf{i},\mathbf{j})\sim (\mathbf{k},\mathbf{l})$,
if $(\mathbf{i},\mathbf{j})$ and $(\mathbf{k},\mathbf{l})$ lie in the same orbit for the associated diagonal action of $\mathfrak{S}_l$ on $I(m|n,l)\times I(m|n,l)$.
Say $(\mathbf{i},\mathbf{j})\in I(m|n,l)\times I(m|n,l)$ is \emph{strict} if $(\bar{i}_s+\bar{j}_s)(\bar{i}_t+\bar{j}_t)=0$ whenever $(\bar{i}_s,\bar{j}_s)=(\bar{i}_t,\bar{j}_t)$ for $1\leq s<t\leq l$.
Denote by $I^2(m|n,l)$ the set of all strict double indexes.
Note $(\mathbf{i},\mathbf{j})$ is strict if and only if the element $\tilde{T}_{\mathbf{i},\mathbf{j}}:=\tilde{T}_{i_1,j_1}...\tilde{T}_{i_l,j_l}$ and
$\tilde{S}_{\mathbf{i},\mathbf{j}}:=\tilde{S}_{i_1,j_1}...\tilde{S}_{i_l,j_l}$ are nonzero.
Denote by $\Omega(m|n,l)$ a fixed set of orbit representatives for the action of $\mathfrak{S}_l$ on $I^2(m|n,l)$.

With $w\in\mathfrak{S}_l$, define $\sigma(\mathbf{i},\mathbf{j};\mathbf{k},\mathbf{l})=
\gamma(\epsilon_{\mathbf{i}}+\epsilon_{\mathbf{j}},w)$
if $(\mathbf{i},\mathbf{j})\cdot w=(\mathbf{k},\mathbf{l})$.
Then $\tilde{T}_{\mathbf{k},\mathbf{l}}=
\sigma(\mathbf{i},\mathbf{j};\mathbf{k},\mathbf{l})\tilde{T}_{\mathbf{i},\mathbf{j}}$.
\subsection{The enhanced natural representation of $\mathrm{GL(m|n)}$ on $\underline{V}$}
\subsubsection{} We adopt some notations from \cite[section 2]{B-K2002}.
Let $V=\mathbb{K}^{m\mid n}$ be the natural representation of $\mathrm{GL(m|n)}$,
the $(m|n)$ dimensional super vector space with canonical basis $v_1,...,v_{m+n}$ where
$\bar{v}_i=\bar{i}$.
Any element of $V\otimes A$ can be written as the form
$\Sigma_{i=1}^{m+n}v_i\otimes a_i$ with all $a_i\in A$.
View the  element of $V\otimes A$ as a column vector $(a_1,...,a_{m+n})^T$,
the transposition of $(a_1,...,a_{m+n})$.
Then the $\mathrm{G}(A)$-action on $V\otimes A$ is the usual one by left multiplication.
The induced comodule structure map $\mathbf{\eta}: V\rightarrow V\otimes k[\mathrm{GL(m|n)}]$ is given by
\begin{equation}\label{eqn1.1}
	\mathbf{\eta}(v_j)=\sum^{m+n}_{i=1}(-1)^{\bar{i}(\bar{i}+\bar{j})}v_i\otimes \tilde{T}_{ij}.
\end{equation}

Let $V^{\otimes l}$ be the $l$th tensor product of $V$,
then $V^{\otimes l}$ becomes an $\mathrm{GL(m|n)}$-module by diagonal action.
The comodule map of $V^{\otimes l}$ induced from  $\mathbf{\eta}$ is still denoted by $\mathbf{\eta}$.
Given $\mathbf{i}\in I(m|n,l)$, let $v_{\mathbf{i}}=v_{i_1}\otimes...\otimes v_{i_l}\in V^{\otimes l}$, giving a basis $\{v_{\mathbf{i}}\}_{\mathbf{i}\in I(m|n,l)}$ for the tensor space $V^{\otimes l}$.
The parity of $v_{\mathbf{i}}$ is $\overline{v}_{i_1}+\overline{v}_{i_2}+...+\overline{v}_{i_l}$.

From the proof of \cite[Lemma 5.1]{B-K2002}, we know the structure map $\mathbf{\eta}:V^{\otimes l}\rightarrow V^{\otimes l}\bigotimes \mathbb{K}[\mathrm{GL(m|n)}]$ satisfies
$$\mathbf{\eta}(v_{\mathbf{j}})=\sum_{\mathbf{i}\in I(m|n,l)}(-1)^{\bar{\mathbf{i}}(\bar{\mathbf{i}}+\bar{\mathbf{j}})}\alpha(\epsilon_{\mathbf{i}}+\epsilon_{\mathbf{j}},\epsilon_{\mathbf{i}})v_{\mathbf{i}}\otimes \tilde T_{\mathbf{i},\mathbf{j}}.$$

\subsubsection{}  We adopt some notations from \cite[3.3]{B-Y-Y2020}.
Let $\underline{V}=V\oplus \mathbb{K}v$ with $v$ being $\mathbb{Z}_2$-homogeneous.
$\underline{V}$ is called the enhanced vector space of $V$ and $v$ is called the enhanced vector of $V$.
There are two possibilities for the parity of $v$:
$\bar{v}=\bar{0}$ or $\bar{v}=\bar{1}$.
The elements in $\underline{V}\otimes A$ can be written as $\sum\limits_{i=1}^{m+n}v_i\otimes b_i+v\otimes b$.
Similarly, we can identify elements of $\underline{V}\otimes A$ with column vectors via
$$\sum\limits_{i=1}^{m+n}v_i\otimes b_i+v\otimes b\longleftrightarrow
\begin{pmatrix}
	b_1 \\
	\vdots \\
	b_m\\
	b\\
	b_{m+1}\\
	\vdots \\
	b_{m+n}
\end{pmatrix}.$$
The action of $\mathrm{\tilde G}(A)$ on $\underline{V}\otimes A$ is the  left multiplication, this is,
\begin{equation}
	(g,a)(\sum_{i=1}^{m+n}v_i\otimes b_i+v\otimes b)=g(\sum_{i=1}^{m+n}v_i\otimes b_i)+v\otimes ab
\end{equation}
The induced comodule map of $\underline{V}$ is
\begin{equation}
	\eta(v_j)=\sum^{m+n}_{i=1}(-1)^{\bar{i}(\bar{i}+\bar{j})}v_i\otimes \tilde{S}_{ij}\text{ for } 1\leq j\leq m+n;
\end{equation}
\begin{equation}
	\eta(v)=v\otimes S.
\end{equation}

There are some notations in the following.

Let $basis(V):=\{v_1,...,v_{m+n}\}$.

$\left|I\right|$ stands for the cardinality of a finite set $I$. For a positive integral $t$, denote by $\underline{t}$ the set $\{1,2,...,t\}.$

From \cite[3.3]{B-Y-Y2020}, the $r$th tensor product of $\underline{V}$ can be decomposed into
\begin{equation}
	\underline{V}^{\otimes r}=\bigoplus_{l=0}^{r}\underline{V}_l^{\otimes r}\\
	=\bigoplus_{l=0}^{r}\bigoplus_{I\subseteq \underline{r},\atop|I|=l}\underline{V}_I^{\otimes r}
\end{equation}
where
$\underline{V}_I^{\otimes r}=\mathrm{Span}_{\mathbb{K}}\{w_{1}\otimes w_{2}\otimes...\otimes w_{r}\mid w_{j}\in basis(V)\text{ if }j\in I\subset\underline{r}\text{, otherwise }w_{j}=v\}$,
and $\underline{V}_l^{\otimes r}=\bigoplus\limits_{I\subseteq \underline{r},\atop|I|=l}\underline{V}_I^{\otimes r}$.

For $I\subset\underline{r}$, let $basis(\underline{V}_I^{\otimes r}):=\{w_{1}\otimes w_{2}\otimes...\otimes w_{r}\mid w_{j}\in basis(V)\text{ if }j\in I\text{, otherwise }w_{j}=v\}$.
%and $basis(\underline{V}_l^{\otimes r}):=\bigcup_{I\subseteq \underline{r},\atop|I|=l}basis(\underline{V}_I^{\otimes r})$.
For a fixed set $I=\{j_1,...,j_l\}\subset\underline{r}$  with $j_1<j_2<...<j_l$ and $\mathbf{i}=(i_1,...,i_l)\in I(m|n,l)$, let $v_{\mathbf{i},I}=w_{1}\otimes ...\otimes w_{r}\in basis(\underline{V}_{I}^{\otimes r})$ satisfying $w_{j_k}=v_{i_k}$ for $1\leq k\leq l$.
In particular, for $0\leq l\leq r$, $v_{\mathbf{i},\underline{l}}=v_{i_1}\otimes v_{i_2}\otimes...\otimes v_{i_l}\otimes\underbrace{v\otimes...\otimes v}_{r-l}\in basis(\underline{V}^{\otimes r}_{\underline{l}})$.

Recall the definition of $\epsilon_{\mathbf{i}}$ for $\mathbf{i}\in I(m|n,l)$. We define $\epsilon_{\mathbf{i},I}=(\bar{w}_1,\bar{w}_2,...,\bar{w}_r)$ for $v_{\mathbf{i},I}=w_{1}\otimes ...\otimes w_{r}$.

\subsection{Schur superalgebras}
Let $A(m|n)$ be the super commutative super algebra generated by $\{\tilde{T}_{ij}\mid 1\leq i,j\leq m+n\}$.
Then $A(m|n)=\bigoplus_{l} A(m|n,l)$ where $A(m|n,l)$ is spanned by $\{\tilde{T}_{\mathbf{i},\mathbf{j}}=\tilde{T}_{i_1,j_1}...\tilde{T}_{i_l,j_l}\mid(\mathbf{i},\mathbf{j})\in\Omega(m|n,l)\}$.
\subsubsection{ }
Then $\eta(V^{\otimes l})\subset V^{\otimes l}\otimes A(m|n,l)$.
Let $S(m|n,l):=\mathrm{Hom}_{\mathbb{K}}(A(m|n,l),\mathbb{K})$,
which is called a \emph{Schur superalgebra} by defining the product $\xi_1\xi_2$ for homogeneous $\xi_1,\xi_2\in A(m|n,l)$ via $(\xi_1\xi_2)(f)=(\xi_1\overline{\otimes} \xi_2)\Delta(f)$ for $f\in A(m|n,l)$.
Then we can view  $V^{\otimes l}$ as a left $S(m|n,l)$-module with the action defined by $\xi(v_{\mathbf{i}}):=(\mathrm{id}_{V^{\otimes r}}\overline{\otimes} \xi)\mathbf{\eta}(v_{\mathbf{i}})$
for $\xi\in S(m|n,l), v_{\mathbf{i}}\in V^{\otimes l}$.

For $(\mathbf{i},\mathbf{j})\in I^2(m|n,l)$, let $\xi_{\mathbf{i},\mathbf{j}}$ be the unique element satisfying
$$\xi_{\mathbf{i},\mathbf{j}}(\tilde{T}_{\mathbf{i},\mathbf{j}})
=\alpha(\epsilon_{\mathbf{i}}+\epsilon_{\mathbf{j}},\epsilon_{\mathbf{i}}+\epsilon_{\mathbf{j}}),\quad
\xi_{\mathbf{i},\mathbf{j}}(\tilde{T}_{\mathbf{k},\mathbf{l}})=0 \ \text{for all }
(\mathbf{i},\mathbf{j})\nsim (\mathbf{k},\mathbf{l}).$$
The elements $\{\xi_{\mathbf{i},\mathbf{j}}\}_{(\mathbf{i},\mathbf{j})\in \Omega(m|n,l)}$ give a basis for $S(m|n,l).$ And the parity of $\xi_{\mathbf{i},\mathbf{j}}$ is $\overline{\mathbf{i}}+\overline{\mathbf{j}}:=\bar{i}_1+...+\bar{i}_l+\bar{j}_1+...+\bar{j}_l$.

Let $e_{\mathbf{i},\mathbf{j}}=e_{i_1,j_1}\otimes...\otimes e_{i_l,j_l}$ where $e_{ij}$ is the $(i,j)$-matrix unit.
Then $e_{\mathbf{i},\mathbf{j}}\in \mathrm{End}_\mathbb{K}(V)^{\otimes l}\cong \mathrm{End}_\mathbb{K}(V^{\otimes l})$ and $e_{\mathbf{i},\mathbf{j}}v_{\mathbf{k}}
=\delta_{\mathbf{j},\mathbf{\mathbf{k}}}\alpha(\epsilon_{\mathbf{i}}+\epsilon_{\mathbf{j}},\epsilon_{\mathbf{\mathbf{j}}})v_{\mathbf{i}}$
for $\mathbf{i},\mathbf{j},\mathbf{\mathbf{k}}\in I(m|n,l)$.

\begin{lemma}$($\cite[Lemma 5.1]{B-K2002}\label{Lemma5.1}$)$
	The representation $\rho_r:S(m|n,r)\rightarrow \mathrm{End}_\mathbb{K}(V^{\otimes r})$ is faithful and satisfies
	$$\rho_r(\xi_{\mathbf{i},\mathbf{j}})=\sum_{(\mathbf{\mathbf{k}},\mathbf{l})\sim (\mathbf{i},\mathbf{j})}\sigma(\mathbf{i},\mathbf{j};\mathbf{k},\mathbf{l})e_{\mathbf{k},\mathbf{l}}$$
	for each $(\mathbf{i},\mathbf{j})\in I^2(m|n,r)$.
	Moreover, for $(\mathbf{k},\mathbf{l}),(\mathbf{i},\mathbf{j})\in I^2(m|n,r)$,
	$$\xi_{\mathbf{i},\mathbf{j}}\xi_{\mathbf{k},\mathbf{l}}=
	\sum_{(\mathbf{s},\mathbf{t})\in \Omega(m\mid n,r)}a_{\mathbf{i},\mathbf{j},\mathbf{k},\mathbf{l},\mathbf{s},\mathbf{t}}
	\xi_{\mathbf{s},\mathbf{t}}$$
	where $a_{\mathbf{i},\mathbf{j},\mathbf{k},\mathbf{l},\mathbf{s},\mathbf{t}}=
	\sum \sigma(\mathbf{i},\mathbf{j};\mathbf{s},\mathbf{h})
	\sigma(\mathbf{k},\mathbf{l};\mathbf{h},\mathbf{t})
	\alpha(\epsilon_{\mathbf{s}}+\epsilon_{\mathbf{h}},\epsilon_{\mathbf{h}}+\epsilon_{\mathbf{t}})$
	summing over all $\mathbf{h}\in I(m|n,r)$ with $(\mathbf{s},\mathbf{h})\sim (\mathbf{i},\mathbf{j}),(\mathbf{h},\mathbf{t})\sim (\mathbf{k},\mathbf{l})$.
\end{lemma}

In the proof of \cite[Lemma 5.1]{B-K2002}, we know that
$$\xi_{\mathbf{i},\mathbf{j}}v_{\mathbf{l}}=\sum_{(\mathbf{k},\mathbf{l})\sim (\mathbf{i},\mathbf{j})}\sigma(\mathbf{i},\mathbf{j};\mathbf{k},\mathbf{l})\alpha(\epsilon_{\mathbf{k}}+\epsilon_{\mathbf{l}},\epsilon_{\mathbf{l}})v_{\mathbf{k}}.$$

\subsubsection{ }
Let $A'(m|n,r):=\mathrm{Span}_{\mathbb{K}}\{\tilde{S}_{\mathbf{i},\mathbf{j}}S^{r-l},S^{r}\mid 1\leq l\leq r,(\mathbf{i},\mathbf{j})\in \Omega(m|n,l)\}.$
Then $\eta(\underline{V}^{\otimes r})\subset \underline{V}^{\otimes r}\otimes A'(m|n,r)$.
Define $S'(m|n,r):=\mathrm{Hom}_{\mathbb{K}}(A'(m|n,r),\mathbb{K})$.

Define $\xi_{\mathbf{i},\mathbf{j},l}\in S'(m|n,r)$ and $\xi_0$ for
$1\leq l\leq r$ and $(\mathbf{i},\mathbf{j})\in I^2(m|n,l)$ satisfying the following conditions:

(1) $\xi_{\mathbf{i},\mathbf{j},l}(S^{r})=0$ and $\xi_{\mathbf{i},\mathbf{j},l}(\tilde{S}_{\mathbf{k},\mathbf{l}}S^{r-d})=0$
for $(\mathbf{i},\mathbf{j})\nsim(\mathbf{k},\mathbf{l})$ or $l\neq d$;

(2) $\xi_{\mathbf{i},\mathbf{j},l}(\tilde{S}_{\mathbf{i},\mathbf{j}}S^{r-l})
=\alpha(\epsilon_{\mathbf{i}}+\epsilon_{\mathbf{j}},
\epsilon_{\mathbf{i}}+\epsilon_{\mathbf{j}})
=\xi_{\mathbf{i},\mathbf{j}}(\tilde{S}_{\mathbf{i},\mathbf{j}})$;

(3) $\xi_{0}(S^{r})=1$ and $\xi_{0}(\tilde{S}_{\mathbf{i},\mathbf{j}}S^{r-l})=0$.

In particular, $\xi_{\mathbf{i},\mathbf{j},r}(\tilde{S}_{\mathbf{i},\mathbf{j}})=\alpha(\epsilon_{\mathbf{i}}+\epsilon_{\mathbf{j}},
\epsilon_{\mathbf{i}}+\epsilon_{\mathbf{j}})$.
Recall that $\Omega(m|n,l)$ denotes a fixed set of orbit representatives for the action of $\mathfrak{S}_l$ on $I^2(m|n,l)$.
Then we can choose $\{\xi_0, \xi_{\mathbf{i},\mathbf{j},l}\mid {1\leq l\leq r,(\mathbf{i},\mathbf{j})\in \Omega(m|n,l)}\}$ as a basis of $S'(m|n,r)$.
For $1\leq l\leq r$ and $(\mathbf{i},\mathbf{j})\in \Omega(m|n,l)$, the parity of $\xi_{\mathbf{i},\mathbf{j},l}$ is still $\overline{\mathbf{i}}+\overline{\mathbf{j}}$. And the parity of $\xi_0$ is $\bar{0}$.
Then $S'(m|n,r)$ is a super vector subspace of $S(m+1|n,r)$ or $S(m|n+1,r)$.

The following equation is independent of the parity of $v$.
\begin{proposition}\label{pro4.2}
	For any $I\subset\underline{r}$ with $|I|=l$, $1\leq l,k\leq r$, $(\mathbf{i},\mathbf{j})\in \Omega(m|n,k)$, and $\mathbf{t}\in I(m\mid n,l)$,
	$$\xi_{\mathbf{i},\mathbf{j},k}v_{\mathbf{t},I}=\delta_{kl}\sum_{(\mathbf{k},\mathbf{t})\sim (\mathbf{i},\mathbf{j})}\sigma(\mathbf{i},\mathbf{j};\mathbf{k},\mathbf{t})\alpha(\epsilon_{\mathbf{k},I}+\epsilon_{\mathbf{t},I},\epsilon_{\mathbf{t},I})v_{\mathbf{k},I}.$$
	Furthermore, we know $	 \xi_{\mathbf{i},\mathbf{j},l}v_{\mathbf{t},\underline{l}}=\xi_{\mathbf{i},\mathbf{j}}(v_{\mathbf{t}})\otimes v^{\otimes (r-l)}$.
\end{proposition}
\begin{proof}
	By a direct computation, we know the structure map $\mathbf{\eta}: \underline{V}^{\otimes r}\longrightarrow\underline{V}^{\otimes r}\otimes A'(m|n,r)$ satisfies
	$$\mathbf{\eta}(v_{\mathbf{t},I})=\sum_{\mathbf{k}\in I(m\mid n,l)}(-1)^{(\overline{\mathbf{k}}+(r-l)\cdot \bar{v})(\overline{\mathbf{k}}+\overline{\mathbf{t}})}\alpha(\epsilon_{\mathbf{k},I}+\epsilon_{\mathbf{t},I},\epsilon_{\mathbf{k},I})v_{\mathbf{k},I}\otimes \tilde S_{\mathbf{k},\mathbf{t}}S^{r-l}.$$
	%By a direct computation, we know the structure map $\mathbf{\eta}: \underline{V}^{\otimes r}\longrightarrow\underline{V}^{\otimes r}\otimes A'(m|n,r)$ satisfies
	%$$\mathbf{\eta}(v_{\mathbf{t},\underline{l}})=\sum_{\mathbf{k}\in I(m|n,l)}(-1)^{(\overline{\mathbf{k}}+(r-l)\cdot \bar{v})(\overline{\mathbf{k}}+\overline{\mathbf{t}}))}\alpha(\epsilon_{\mathbf{k}}+\epsilon_{\mathbf{t}},\epsilon_{\mathbf{k}})v_{\mathbf{k},\underline{l}}\otimes \tilde S_{\mathbf{k},\mathbf{t}}S^{r-l}.$$
	Then we get
	\begin{align*}
		\xi_{\mathbf{i},\mathbf{j},k}v_{\mathbf{t},I}
		&=(id\overline{\otimes}\xi_{\mathbf{i},\mathbf{j},k})\mathbf{\eta}(v_{\mathbf{t},I})\\
		&=\sum_{\mathbf{k}\in I(m|n,l)}(-1)^{(\overline{\mathbf{k}}+(r-l)\cdot \bar{v})(\overline{\mathbf{k}}+\overline{\mathbf{t}})}(-1)^{(\overline{\mathbf{i}}+\overline{\mathbf{j}})(\overline{\mathbf{k}}+(r-l)\bar{v})}\\
		 &\qquad\alpha(\epsilon_{\mathbf{k},I}+\epsilon_{\mathbf{t},I},\epsilon_{\mathbf{k},I})v_{\mathbf{k},I}\xi_{\mathbf{i},\mathbf{j},k}(\tilde S_{\mathbf{k},\mathbf{t}}S^{r-l})\\
		&=\delta_{kl}\sum_{(\mathbf{k},\mathbf{t})\sim (\mathbf{i},\mathbf{j})}\sigma(\mathbf{i},\mathbf{j};\mathbf{k},\mathbf{t})\alpha(\epsilon_{\mathbf{k},I}+\epsilon_{\mathbf{t},I},\epsilon_{\mathbf{t},I})v_{\mathbf{k},I}.\\
		%	&=\xi_{\mathbf{i},\mathbf{j}}(v_{\mathbf{t}})\otimes v^{\otimes r}
	\end{align*}		
\end{proof}
Recall that $e_{ij}$ is the $(i,j)$-matrix unit.
Let $$e'_{ij}=
\begin{cases}
	&e_{ij}, \text{ if } 1\leq i,j\leq m \\
	&e_{i,j-m},\text{ if } 1\leq i\leq m \text{ and } m+1\leq j\leq m+n\\
	&e_{i-m,j}, \text{ if } m+1\leq i\leq m+n \text{ and } 1\leq j\leq m\\
	&e_{i-m,j-m}, \text{ if }  m+1\leq i,j\leq m+n
\end{cases}$$
and $e':=e_{m+1,m+1}$.
For $\mathbf{i},\mathbf{j}\in I(m|n,l)$ and $I\subset\underline{r}$ with $|I|=l$, define $e_{\mathbf{i},\mathbf{j},I}=e''_{1}\otimes e''_{2}\otimes...\otimes e''_{r}$ where $e''_k=e'_{i_k,j_k}$ if $k\in I$ and $e''_k=e'=e_{m+1,m+1}$ if $k\notin I$.
Then $e_{\mathbf{i},\mathbf{j},I}v_{\mathbf{k},J}
=\delta_{\mathbf{j},\mathbf{\mathbf{k}}}\delta_{IJ}\alpha(\epsilon_{\mathbf{i},I}+\epsilon_{\mathbf{j},I},\epsilon_{\mathbf{\mathbf{k}},I})v_{\mathbf{i},I}$ and $e_{\mathbf{i},\mathbf{j},I}e_{\mathbf{k},\mathbf{l},J}=\delta_{IJ}\delta_{\mathbf{j},\mathbf{k}}\alpha(\epsilon_{\mathbf{i}}+\epsilon_{\mathbf{j}},\epsilon_{\mathbf{k}}+\epsilon_{\mathbf{l}})e_{\mathbf{i},\mathbf{l},I}$.

\begin{lemma}\label{LLe}
	The representation $\rho_r:S'(m|n,r)\rightarrow \mathrm{End}_\mathbb{K}(\underline{V}^{\otimes r})$ is faithful and satisfies
	$$\rho_r(\xi_{\mathbf{i},\mathbf{j},l})=\sum_{I\subset\underline{r}\atop |I|=l}\sum_{(\mathbf{\mathbf{k}},\mathbf{l})\sim (\mathbf{i},\mathbf{j})}\sigma(\mathbf{i},\mathbf{j};\mathbf{k},\mathbf{l})e_{\mathbf{k},\mathbf{l},I}$$
	for each $(\mathbf{i},\mathbf{j})\in I^2(m|n,l)$ with $1\leq l\leq r$.
	Moreover, for $(\mathbf{k},\mathbf{l}),(\mathbf{i},\mathbf{j})\in I^2(m|n,l)$,
	$$\xi_{\mathbf{i},\mathbf{j},k}\xi_{\mathbf{k},\mathbf{l},l}=
	\delta_{kl}\sum_{(\mathbf{s},\mathbf{t})\in \Omega(m|n,l)}a_{\mathbf{i},\mathbf{j},\mathbf{k},\mathbf{l},\mathbf{s},\mathbf{t}}
	\xi_{\mathbf{s},\mathbf{t},l}$$
	where $a_{\mathbf{i},\mathbf{j},\mathbf{k},\mathbf{l},\mathbf{s},\mathbf{t}}=
	\sum \sigma(\mathbf{i},\mathbf{j};\mathbf{s},\mathbf{h})
	\sigma(\mathbf{k},\mathbf{l};\mathbf{h},\mathbf{t})
	\alpha(\epsilon_{\mathbf{s}}+\epsilon_{\mathbf{h}},\epsilon_{\mathbf{h}}+\epsilon_{\mathbf{t}})$
	summing over all $\mathbf{h}\in I(m|n,l)$ with $(\mathbf{s},\mathbf{h})\sim (\mathbf{i},\mathbf{j}),(\mathbf{h},\mathbf{t})\sim (\mathbf{k},\mathbf{l})$.
\end{lemma}
\begin{proof}
	By Proposition \ref{pro4.2},
	\begin{align*}
	\rho_r(\xi_{\mathbf{i},\mathbf{j},k})v_{\mathbf{l},I}&=\delta_{k,|I|}\sum_{(\mathbf{k},\mathbf{l})\sim (\mathbf{i},\mathbf{j})}\sigma(\mathbf{i},\mathbf{j};\mathbf{k},\mathbf{l})\alpha(\epsilon_{\mathbf{k},I}+\epsilon_{\mathbf{l},I},\epsilon_{\mathbf{l},I})v_{\mathbf{k},I}\\
	&=\delta_{k,|I|}\sum_{(\mathbf{k},\mathbf{l})\sim (\mathbf{i},\mathbf{j})}\sigma(\mathbf{i},\mathbf{j};\mathbf{k},\mathbf{l})e_{\mathbf{k},\mathbf{l},I}v_{\mathbf{l},I}\\
	&=\sum_{J\subset\underline{r},\atop |J|=k}\sum_{(\mathbf{k},\mathbf{l})\sim (\mathbf{i},\mathbf{j})}\sigma(\mathbf{i},\mathbf{j};\mathbf{k},\mathbf{l})e_{\mathbf{k},\mathbf{l},J}v_{\mathbf{l},I}.
	\end{align*}
	So $\rho_r(\xi_{\mathbf{i},\mathbf{j},k})=\sum_{J\subset\underline{r},\atop |J|=k}\sum_{(\mathbf{k},\mathbf{l})\sim (\mathbf{i},\mathbf{j})}\sigma(\mathbf{i},\mathbf{j};\mathbf{k},\mathbf{l})e_{\mathbf{k},\mathbf{l},J}.$
	
	Moreover,
		$$\rho_r(\xi_0)v_{\mathbf{l},I}=\sum_{\mathbf{k}\in I(m|n,l)}(-1)^{(\overline{\mathbf{k}}+(r-l)\cdot \bar{v})(\overline{\mathbf{k}}+\overline{\mathbf{l}})}
		 \alpha(\epsilon_{\mathbf{k},I}+\epsilon_{\mathbf{l},I},\epsilon_{\mathbf{k},I})v_{\mathbf{k},I}\xi_{0}(\tilde S_{\mathbf{k},\mathbf{l}}S^{r-l})=0$$
and
	$$\rho_r(\xi_0)v^{\otimes r}=(id\overline{\otimes}\xi_0)\mathbf{\eta}(v^{\otimes r})
	=(id\overline{\otimes}\xi_0)(v^{\otimes r}\otimes S^r)
	=v^{\otimes r}\xi_0(S^r)
	=v^{\otimes r}.$$
Therefore $\rho_r(\xi_0)=(e')^{\otimes r}$.

	According to the definition of $e_{\mathbf{k},\mathbf{l},J}$, we know that $\{\rho_r(\xi_0),\rho_r(\xi_{\mathbf{i},\mathbf{j},k})\mid 1\leq k\leq r,\mathbf{i},\mathbf{j}\in\Omega(m|n,l)\}$ are linearly independent. Therefore $\rho_r$ is faithful.
	
		Since $e_{\mathbf{i},\mathbf{j},I}e_{\mathbf{k},\mathbf{l},J}=\delta_{IJ}\delta_{\mathbf{j},\mathbf{k}}\alpha(\epsilon_{\mathbf{i}}+\epsilon_{\mathbf{j}},\epsilon_{\mathbf{k}}+\epsilon_{\mathbf{l}})e_{\mathbf{i},\mathbf{l},I}$ and $\rho_r$ is faithful, it follows that $\xi_{\mathbf{i},\mathbf{j},k}\xi_{\mathbf{k},\mathbf{l},l}=0$ for $k\neq l$.
If $k=l$, then
\begin{align*}
	&\rho_r(\xi_{\mathbf{i},\mathbf{j},l})\rho_r(\xi_{\mathbf{k},\mathbf{l},l})\\
	&=\sum_{J\subset\underline{r}\atop\mid J\mid =l}\sum_{I\subset\underline{r}\atop\mid I\mid =l}\sum_{(\mathbf{s},\mathbf{h})\sim (\mathbf{i},\mathbf{j})\atop(\mathbf{h}',\mathbf{t})\sim (\mathbf{k},\mathbf{l})}\sigma(\mathbf{i},\mathbf{j};\mathbf{s},\mathbf{h})\sigma(\mathbf{k},\mathbf{l};\mathbf{h}',\mathbf{t})e_{\mathbf{s},\mathbf{h},I}e_{\mathbf{h}',\mathbf{t},I}\\
	&=\sum_{I\subset\underline{r}\atop\mid I\mid =l}\sum_{(\mathbf{s},\mathbf{h})\sim (\mathbf{i},\mathbf{j})\atop(\mathbf{h},\mathbf{t})\sim (\mathbf{k},\mathbf{l})}\sigma(\mathbf{i},\mathbf{j};\mathbf{s},\mathbf{h})\sigma(\mathbf{k},\mathbf{l};\mathbf{h},\mathbf{t})\alpha(\epsilon_{\mathbf{s}}+\epsilon_{\mathbf{h}},\epsilon_{\mathbf{h}}+\epsilon_{\mathbf{t}})e_{\mathbf{s},\mathbf{t},I}\\
	&=\sum_{(\mathbf{s},\mathbf{t})\in \Omega(m\mid n,l)}\sum_{I\subset\underline{r}\atop\mid I\mid =l}
	\sum_{(\mathbf{s}',\mathbf{t}')\sim(\mathbf{s},\mathbf{t})\atop{(\mathbf{s}',\mathbf{h}')\sim (\mathbf{i},\mathbf{j})\atop(\mathbf{h}',\mathbf{t}')\sim (\mathbf{k},\mathbf{l})}}
	 \sigma(\mathbf{i},\mathbf{j};\mathbf{s}',\mathbf{h}')\sigma(\mathbf{k},\mathbf{l};\mathbf{h}',\mathbf{t}')
	 \alpha(\epsilon_{\mathbf{s}'}+\epsilon_{\mathbf{h}'},\epsilon_{\mathbf{h}'}+\epsilon_{\mathbf{t}'})e_{\mathbf{s}',\mathbf{t}',I}\\
	&=\sum_{(\mathbf{s},\mathbf{t})\in \Omega(m\mid n,l)}
	\sum_{(\mathbf{s},\mathbf{h})\sim (\mathbf{i},\mathbf{j})\atop(\mathbf{h},\mathbf{t})\sim (\mathbf{k},\mathbf{l})}
	 \sigma(\mathbf{i},\mathbf{j};\mathbf{s},\mathbf{h})\sigma(\mathbf{k},\mathbf{l};\mathbf{h},\mathbf{t})
	\alpha(\epsilon_{\mathbf{s}}+\epsilon_{\mathbf{h}},\epsilon_{\mathbf{h}}+\epsilon_{\mathbf{t}})\\
	&\qquad\sum_{I\subset\underline{r}\atop\mid I\mid =l}\sum_{(\mathbf{s}',\mathbf{t}')\sim(\mathbf{s},\mathbf{t})}
	\sigma(\mathbf{s},\mathbf{t};\mathbf{s}',\mathbf{t}')e_{\mathbf{s}',\mathbf{t}',I}\\
	&=\sum_{(\mathbf{s},\mathbf{t})\in \Omega(m\mid n,l)}
	 a_{\mathbf{i},\mathbf{j},\mathbf{k},\mathbf{l},\mathbf{s},\mathbf{t}}\rho_r(\xi_{\mathbf{s},\mathbf{t},l}).
\end{align*}
Since $\rho_r$ is faithful, it follows that $\xi_{\mathbf{i},\mathbf{j},l}\xi_{\mathbf{k},\mathbf{l},l}=
\sum_{(\mathbf{s},\mathbf{t})\in \Omega(m\mid n,l)}a_{\mathbf{i},\mathbf{j},\mathbf{k},\mathbf{l},\mathbf{s},\mathbf{t}}
\xi_{\mathbf{s},\mathbf{t},l}$.

\end{proof}

According to Lemma \ref{LLe} and Lemma \ref{Lemma5.1}, we know that $S'(m|n,r)$ is a super subalgebra of $S(m+1|n,r)$ or $S(m|n+1,r)$.
We also can embed $S(m|n,l)$ into $S'(m|n,r)$ by $\xi_{\mathbf{i},\mathbf{j}}\longmapsto\xi_{\mathbf{i},\mathbf{j},l}$ for all $(\mathbf{i},\mathbf{j})\in I^2(m|n,l)$, and denote the injective homomorphism by $\alpha_l$.
Let $S(m|n,0):=\mathbb{K}\xi_0$ and $\alpha_0(\xi_0)=\xi_0$,
then we have $S'(m|n,r)=\sum_{l=0}^r\alpha_l(S(m|n,l))\cong\sum_{l=0}^rS(m|n,l)$.
The action of $S(m|n,l)$ on $\underline{V}^{\otimes r}$ is induced by $\alpha_l$.
By Proposition \ref{pro4.2}, we know that $S(m|n,l)(\underline{V}^{\otimes r}_I)\subset \underline{V}^{\otimes r}_I$
and $S(m|n,l)(\underline{V}^{\otimes r}_k)=0$ for all $k\neq l$.

\subsection{Duality between Schur superalgebras $S(m|n,r)$ and group algebras $\mathbb{K}\mathfrak{S}_r$}

Denote by $\pi_r$ the representation of symmetric group $\mathfrak{S}_r$ on $V^{\otimes r}$ by letting
$$v_{\mathbf{i}}(j,j+1)=(-1)^{\bar{i}_{j}\bar{i}_{j+1}}v_{i_1}\otimes...\otimes v_{i_{j+1}}\otimes v_{i_j}\otimes...\otimes v_{i_r}$$
for each $\mathbf{i}\in I(m|n,r)$ and each $1\leq j<r$.
For arbitrary $w\in \mathfrak{S}_r$, we have that
$$v_{\mathbf{i}}\cdot w=\gamma(\epsilon_{\mathbf{i}},w)v_{\mathbf{i}\cdot w}.$$

\begin{theorem}\label{theorem5.2}
	$(1)$ $($\cite[Theorem 5.2]{B-K2002}$)$ $\rho_r:S(m|n,r)\rightarrow \mathrm{End}_{\mathbb{K}\mathfrak{S}_r}(V^{\otimes r})$ is an isomorphism.
	
	$(2)$ $($\cite[Remark 5.8(iii)]{B-K2002}$)$ $\mathrm{End}_{S(m\mid n,r)}(V^{\otimes r})\cong \mathbb{K}\mathfrak{S}_r$ when $r\leq m+n$.
\end{theorem}

\section{Weak degenerate double Hecke algebras}
In this section, we will introduce weak degenerate double Hecke algebras. The definition of weak degenerate double Hecke algebras is a weak version of degenerate double Hecke algebras introduced in \cite[section 3]{B-Y-Y2020} with reducing two relations of generators.
\subsection{Weak degenerate double Hecke algebras}
Let $s_i=(i,i+1)\in \mathfrak{S}_r$ for $i=1,...,r-1$.
\begin{definition}
	The $l$th weak degenerate double Hecke algebra $\underline{\mathcal{H}}^{l}_{r}$ of $\mathfrak{S}_r$ is an infinite dimensional associative algebra with generators $\{\mathbf{x}_\sigma\mid\sigma\in \mathfrak{S}_l\}\cup\{\mathbf{s}_i\mid i=1,2,...,r-1\}$ and the following relations
	\begin{equation}\label{h1}
		\mathbf{s}_i^2=1, \mathbf{s}_i\mathbf{s}_j=\mathbf{s}_j\mathbf{s}_i \text{ for }1\leq i\neq j\leq r-1,|j-i|>1;
	\end{equation}
	\begin{equation}\label{h2}
		\mathbf{s}_i\mathbf{s}_j\mathbf{s}_i=\mathbf{s}_j\mathbf{s}_i\mathbf{s}_j \text{ for }1\leq i\neq j\leq r-1,|j-i|=1;
	\end{equation}
	\begin{equation}\label{h3}
		\mathbf{x}_\sigma\mathbf{x}_\mu=\mathbf{x}_{\sigma\circ\mu} \text{ for }\sigma,\mu\in \mathfrak{S}_l;
	\end{equation}
	\begin{equation}\label{h4}
		 \mathbf{s}_i\mathbf{x}_\sigma=\mathbf{x}_{s_i\circ\sigma},\mathbf{x}_\sigma\mathbf{s}_i=\mathbf{x}_{\sigma\circ s_i}\text{ for }\sigma\in\mathfrak{S}_l,i<l;
	\end{equation}
	\begin{equation}\label{h5}
		\mathbf{s}_i\mathbf{x}_\sigma=\mathbf{x}_\sigma\mathbf{s}_i\text{ for }\sigma\in\mathfrak{S}_l,i>l.
	\end{equation}
\end{definition}
The above definition comes from \cite[3.1.1]{B-Y-Y2020} with a bit change that (\ref{h5}) in \cite[3.1.1]{B-Y-Y2020} is $\mathbf{s}_i\mathbf{x}_\sigma=\mathbf{x}_\sigma=\mathbf{x}_\sigma\mathbf{s}_i$.
\begin{definition}$($\cite[Definition 3.1]{B-Y-Y2020}$)$
	The weak degenerate double Hecke algebra $\underline{\mathcal{H}}_{r}$ of $\mathfrak{S}_r$ is an associative algebra with generators $\mathbf{s}_i(i=1,...,r-1)$, and $\mathbf{x}_\sigma^{(l)}$ for $\sigma\in \mathfrak{S}_l,l=0,1,...,r$
	and with relations as (\ref{h1})-(\ref{h5}) in which $\mathbf{x}_\sigma$, $\mathbf{x}_\mu$ are replaced by $\mathbf{x}_\sigma^{(l)}$, $\mathbf{x}_\mu^{(l)}$,
	and additional ones:
	\begin{equation}\label{h6}
		\mathbf{x}_\delta^{(l)}\mathbf{x}_\gamma^{(k)}=0\text{ for }\delta\in\mathfrak{S}_l, \gamma\in \mathfrak{S}_k\text{ with }k\neq l.
	\end{equation}
\end{definition}

\subsection{The right action of $\underline{\mathcal{H}}^l_{r}$ on $\underline{V}_l^{\otimes r}$}

We will give some new symbols.
Let $\mathcal{N}:=\{1,2,..,m+n,m+n+1\}$.
Let $v'_i=v_i$ if $i=1,..,m$, $v'_{m+1}=v$ and $v'_{i+1}=v_i$ if $i=m+1,...,m+n$.
For any $\mathbf{j}=(j_1,...,j_r)\in\mathcal{N}^r$, let $v'_{\mathbf{j}}=v'_{j_1}\otimes v'_{j_2}\otimes ...\otimes v'_{j_r}\in\underline{V}^{\otimes r}$. Let $\epsilon'_{\mathbf{j}}=(\bar{v}'_{j_1},\bar{v}'_{j_2},...,\bar{v}'_{j_r})\in\mathbb{Z}^{r}$.

Recall that $\pi_l$ is the representation of symmetric group $\mathfrak{S}_l$ on $V^{\otimes l}$.
Let $x_\sigma^{\underline{l}}=\pi_l(\sigma)\otimes id^{\otimes r-l}\in \mathrm{End}_{\mathbb{K}}(\underline{V}_{\underline{l}}^{\otimes r})$,
and extend the action of $x_\sigma^{\underline{l}}$ on $\underline{V}_l^{\otimes r}$ by
defining $v_{\mathbf{i},I}\cdot x_\sigma^{\underline{l}}=\delta_{\underline{l},I}\gamma(\epsilon_{\mathbf{i}},\sigma)v_{\mathbf{i}\cdot\sigma,\underline{l}}$ for all $v_{\mathbf{i},I}\in basis(\underline{V}_l^{\otimes r}):=\bigcup_{I\subseteq \underline{r},\atop|I|=l}basis(\underline{V}_I^{\otimes r})$.

For each $v'_{\mathbf{j}}\in basis(\underline{V}^{\otimes r}_l)$, define
$$v'_{\mathbf{j}}\cdot \mathbf{s}_i=\pi_r(s_i)v'_{\mathbf{j}}=\gamma(\epsilon'_{\mathbf{j}}, s_i)v'_{\mathbf{j}\cdot s_i},$$
$$v'_{\mathbf{j}}\cdot\mathbf{x}^{{(l)}}_\sigma=v'_{\mathbf{j}}\cdot x^{\underline{l}}_\sigma=0\text{ if }v'_{\mathbf{j}}\notin \underline{V}^{\otimes r}_{\underline{l}},$$
$$v'_{\mathbf{j}}\cdot\mathbf{x}^{{(l)}}_\sigma=v'_{\mathbf{j}}\cdot x^{\underline{l}}_\sigma=
\gamma(\epsilon_{\mathbf{i}},\sigma)v_{\mathbf{i}\cdot\sigma,\underline{l}}
\text{ if }v'_{\mathbf{j}}=v_{\mathbf{i},\underline{l}}\in\underline{V}^{\otimes r}_{\underline{l}}.$$

\begin{lemma}\label{l1}
	The following statements hold.
	
	$(1)$ The linear map $\Xi_l:\underline{\mathcal{H}}_{r}^l\rightarrow \mathrm{End}(\underline{V}_{l}^{\otimes r})_{\bar{0}}$ defined as above  is a homomorphism of algebras.
	
	$(2)$ For any $l\in \{0,1,...,r\}$, $\xi\in S'(m|n,r)$ commutes with any elements from $\Xi_l(\underline{\mathcal{H}}_{r}^l)$ in $\mathrm{End}_{\mathbb{K}}(\underline{V}^{\otimes r}_l)$.
	
	$(3)$ The linear map of $\Xi:\underline{\mathcal{H}}_{r}\rightarrow \mathrm{End}(\underline{V}^{\otimes r})_{\bar{0}}$ defined via
	\begin{equation}
		\Xi|_{\mathbb{K}\mathfrak{S}_r}=\pi_r,
	\end{equation}
	\begin{equation}
		\Xi(\mathbf{x}_\sigma^{(l)})|_{\underline{V}_l^{\otimes r}}=x_\sigma^{\underline{l}},
	\end{equation}
	\begin{equation}
		\Xi(\mathbf{x}_\sigma^{(l)})|_{\underline{V}_k^{\otimes r}}=0\text{ for }k\neq l
	\end{equation}
is a homomorphism of algebras.
\end{lemma}
\begin{proof}
	(1) For $l\leq r$, we need to show that $\Xi_l$ keep the relations (\ref{h1})-(\ref{h5}).
	
	Let $v'_{\mathbf{t}}=v'_{t_1}\otimes v'_{t_2}\otimes...\otimes v'_{t_r}\in V^{\otimes r}_{l}$.
	
	For (\ref{h1}), if $1\leq i\leq r-1$, then
	\begin{align*}
		&\quad v'_{\mathbf{t}}\cdot\mathbf{s}_i^2\\
		&= (-1)^{\bar{v}'_{t_i}\bar{v}'_{t_{i+1}}}(v'_{t_1}\otimes...\otimes v'_{t_{i+1}}\otimes v'_{t_i}\otimes...\otimes v'_{t_r})\cdot \mathbf{s}_i\\
		 &=(-1)^{\bar{v}'_{t_i}\bar{v}'_{t_{i+1}}}(-1)^{\bar{v}'_{t_i}\bar{v}'_{t_{i+1}}}v'_{t_1}\otimes...\otimes v'_{t_{i}}\otimes v'_{t_{i+1}}\otimes...\otimes v'_{t_r}\\
		&=v'_{\mathbf{t}}.
	\end{align*}
	If $j-i>1$, then
	\begin{align*}
		&\quad v'_{\mathbf{t}}\cdot \mathbf{s}_i\mathbf{s}_j\\
		&=(-1)^{\bar{v}'_{{t}_i}\bar{v}'_{t_{i+1}}}(v'_{t_1}\otimes...\otimes v'_{t_{i+1}}\otimes v'_{t_i}\otimes...\otimes v'_{t_r})\cdot \mathbf{s}_j\\
		 &=(-1)^{\bar{v}'_{{t}_i}\bar{v}'_{t_{i+1}}}(-1)^{\bar{v}'_{t_j}\bar{v}'_{t_{j+1}}}(v'_{t_1}\otimes...\otimes v'_{t_{i+1}}\otimes v'_{t_i}\otimes...\otimes v'_{t_{j+1}}\otimes v'_{t_j}\otimes...\otimes v'_{t_r})\\
		&=v'_{\mathbf{t}}\cdot \mathbf{s}_j\mathbf{s}_i.
	\end{align*}
	
	 For (\ref{h2}), if $j-i=1$, then
	\begin{align*}
		&\quad v'_{\mathbf{t}}\cdot \mathbf{s}_i\mathbf{s}_j\mathbf{s}_i\\
		&=(-1)^{\bar{v}'_{t_i}\bar{v}'_{{t}_{i+1}}}(v'_{t_1}\otimes...\otimes v'_{t_{i+1}}\otimes v'_{t_i}\otimes...\otimes v'_{t_r})\cdot \mathbf{s}_j\mathbf{s}_i\\
		 &=(-1)^{\bar{v}'_{t_i}\bar{v}'_{{t}_{i+1}}}(-1)^{\bar{v}'_{t_i}\bar{v}'_{{t}_{i+2}}}(v'_{t_1}\otimes...\otimes v'_{t_{i+1}}\otimes v'_{t_{i+2}}\otimes v'_{t_{i}}\otimes...\otimes v'_{t_r})\cdot \mathbf{s}_i\\
		 &=(-1)^{\bar{v}'_{t_i}\bar{v}'_{{t}_{i+1}}}(-1)^{\bar{v}'_{t_i}\bar{v}'_{{t}_{i+2}}}(-1)^{\bar{v}'_{t_{i+2}}\bar{v}'_{{t}_{i+1}}}(v'_{t_1}\otimes...\otimes v'_{t_{i+2}}\otimes v'_{t_{i+1}}\otimes v'_{t_{i}}\otimes...\otimes v'_{t_r}),
	\end{align*}
and
	\begin{align*}
		&\quad v'_{\mathbf{t}}\cdot \mathbf{s}_j\mathbf{s}_i\mathbf{s}_j\\
		&=(-1)^{\bar{v}'_{t_{i+1}}\bar{v}'_{{t}_{i+2}}}(v'_{t_1}\otimes...\otimes v'_{t_{i+2}}\otimes v'_{t_{i+1}}\otimes...\otimes v'_{t_r})\cdot \mathbf{s}_i\mathbf{s}_j\\
		 &=(-1)^{\bar{v}'_{t_{i+1}}\bar{v}'_{{t}_{i+2}}}(-1)^{\bar{v}'_{t_{i}}\bar{v}'_{{t}_{i+2}}}(v'_{t_1}\otimes...\otimes v'_{t_{i+2}}\otimes v'_{t_{i}}\otimes v'_{t_{i+1}}\otimes...\otimes v'_{t_r})\cdot \mathbf{s}_j\\
		 &=(-1)^{\bar{v}'_{t_{i+1}}\bar{v}'_{{t}_{i+2}}}(-1)^{\bar{v}'_{t_{i}}\bar{v}'_{{t}_{i+2}}}(-1)^{\bar{v}'_{t_{i}}\bar{v}'_{{t}_{i+1}}}(v'_{t_1}\otimes...\otimes v'_{t_{i+2}}\otimes v'_{t_{i+1}}\otimes v'_{t_{i}}\otimes...\otimes v'_{t_r}).
	\end{align*}
	
	Let $v_{\mathbf{t},\underline{l}}=v_{t_1}\otimes v_{t_2}\otimes...\otimes v_{t_l}\otimes v\otimes...\otimes v\in V^{\otimes r}_{\underline{l}}$.
	
 For (\ref{h3}),	
\begin{align*}
	&\quad v_{\mathbf{t},\underline{l}}\cdot\mathbf{x}_\sigma\mathbf{x}_\mu\\
	&=\gamma(\epsilon_{\mathbf{t}},\sigma)v_{\mathbf{t}\cdot\sigma,\underline{l}}\cdot\mathbf{x}_\mu\\
	 &=\gamma(\epsilon_{\mathbf{t}},\sigma)\gamma(\epsilon_{\mathbf{t}\cdot\sigma},\mu)v_{\mathbf{t}\cdot\sigma\mu,\underline{l}}\\
	&=\gamma(\epsilon_{\mathbf{t}},\sigma\mu)v_{\mathbf{t}\cdot\sigma\mu,\underline{l}}\\
	&=v_{\mathbf{t},\underline{l}}\cdot\mathbf{x}_{\sigma\mu}.
\end{align*}
	
 For (\ref{h4}),	if $i<l$ and $\sigma\in\mathfrak{S}_l$, then
	\begin{align*}
		&\quad v_{\mathbf{t},\underline{l}}\cdot\mathbf{s}_i\mathbf{x}_\sigma\\
		&=\gamma(\epsilon_{\mathbf{t}},s_i)v_{\mathbf{t}\cdot s_i,\underline{l}}\cdot\mathbf{x}_\sigma\\
		&=\gamma(\epsilon_{\mathbf{t}},s_i)\gamma(\epsilon_{\mathbf{t}}\cdot s_i,\sigma)v_{\mathbf{t}\cdot s_i\sigma,\underline{l}}\\
		&=\gamma(\epsilon_{\mathbf{t}},s_i\sigma)v_{\mathbf{t}\cdot s_i\sigma,\underline{l}}\\
		&=v_{\mathbf{t},\underline{l}}\cdot\mathbf{x}_{s_i\sigma},
	\end{align*}
and
	\begin{align*}
		&\quad v_{\mathbf{t},\underline{l}}\cdot\mathbf{x}_\sigma\mathbf{s}_i\\
		&=\gamma(\epsilon_{\mathbf{t}},\sigma)v_{\mathbf{t}\cdot \sigma,\underline{l}}\cdot\mathbf{s}_i\\
		&=\gamma(\epsilon_{\mathbf{t}},\sigma)\gamma(\epsilon_{\mathbf{t}}\cdot \sigma,s_i)v_{\mathbf{t}\cdot \sigma s_i,\underline{l}}\\
		&=\gamma(\epsilon_{\mathbf{t}},\sigma s_i)v_{\mathbf{t}\cdot \sigma s_i,\underline{l}}\\
		&=v_{\mathbf{t},\underline{l}}\cdot\mathbf{x}_{s_i\sigma}.
	\end{align*}
	
 For (\ref{h5}), if $i>l$ ,$\sigma\in\mathfrak{S}_l$, and $\bar{v}=\bar{0}$, then
	\begin{align*}
		&\quad  v_{\mathbf{t},\underline{l}}\cdot\mathbf{s}_i\mathbf{x}_\sigma\\
		&=v_{\mathbf{t},\underline{l}}\cdot\mathbf{x}_\sigma\\
		&=\gamma(\epsilon_{\mathbf{t}},\sigma)v_{\mathbf{t}\cdot \sigma,\underline{l}}\\
		&=\gamma(\epsilon_{\mathbf{t}},\sigma)v_{\mathbf{t}\cdot \sigma,\underline{l}}\cdot \mathbf{s}_i\\
		&=v_{\mathbf{t},\underline{l}}\cdot\mathbf{x}_\sigma\mathbf{s}_i.
	\end{align*}
      If $i>l$ ,$\sigma\in\mathfrak{S}_l$, and $\bar{v}=\bar{1}$, then
    \begin{align*}
        &\quad  v_{\mathbf{t},\underline{l}}\cdot\mathbf{s}_i\mathbf{x}_\sigma\\
    	&=-v_{\mathbf{t},\underline{l}}\cdot\mathbf{x}_\sigma\\
    	&=-\gamma(\epsilon_{\mathbf{t}},\sigma)v_{\mathbf{t}\cdot \sigma,\underline{l}}\\
    	&=\gamma(\epsilon_{\mathbf{t}},\sigma)v_{\mathbf{t}\cdot \sigma,\underline{l}}\cdot \mathbf{s}_i\\
    	&=v_{\mathbf{t},\underline{l}}\cdot\mathbf{x}_\sigma\mathbf{s}_i.
   \end{align*}
	(2)
	If $\overline{v}=\overline{0}$, then $S'(m|n,r)\subset S(m+1|n,r)\cong \mathrm{End}_{\mathbb{K}\mathfrak{S}_r}(\underline{V}^{\otimes r}).$
	If $\overline{v}=\overline{1}$, then $S'(m|n,r)\subset S(m|n+1,r)\cong \mathrm{End}_{\mathbb{K}\mathfrak{S}_r}(\underline{V}^{\otimes r}).$
	So the elements of $S'(m|n,r)$ commutes with $\Xi(\mathbf{s}_i) \text{ for } 1\leq i\leq r-1$ without the consideration of the parity of $v$.
	
	As a result, we only need to prove $\xi_{\mathbf{i},\mathbf{j},l}\mathbf{x}_\sigma=\mathbf{x}_\sigma\xi_{\mathbf{i},\mathbf{j},l}$ for $\sigma\in\mathfrak{S}_l$.
	
	We have
		\begin{align*}
		&\ \ \ \xi_{\mathbf{i},\mathbf{j},l}(v_{\mathbf{j}',\underline{l}}\cdot \mathbf{x}_\sigma)
		=\xi_{\mathbf{i},\mathbf{j},l}(v_{\mathbf{j}'\cdot \sigma,\underline{l}})\gamma(\epsilon_{\mathbf{j}'},\sigma)\\
		&=\sum_{(\mathbf{i},\mathbf{j})\sim (\mathbf{k},\mathbf{j}'\cdot \sigma)}
		\sigma(\mathbf{i},\mathbf{j};\mathbf{k},\mathbf{j}'\cdot \sigma)
		\alpha(\epsilon_{\mathbf{k}}+\epsilon_{\mathbf{j}'\sigma},\epsilon_{\mathbf{j}'\sigma})
		\gamma(\epsilon_{\mathbf{j}'},\sigma)v_{\mathbf{k}}\otimes v^{\otimes r-l},
	\end{align*}	
    and
	\begin{align*}
		&\ \ \ (\xi_{\mathbf{i},\mathbf{j},l}v_{\mathbf{j'},l})\mathbf{x}_\sigma\\
		&=\sum_{(\mathbf{k},\mathbf{j'})\sim (\mathbf{i},\mathbf{j})}
		\sigma(\mathbf{i},\mathbf{j};\mathbf{k},\mathbf{j'})
		\alpha(\epsilon_{\mathbf{k}}+\epsilon_{\mathbf{j'}},\epsilon_{\mathbf{j'}})
		(v_{\mathbf{k}}\otimes v^{\otimes r-l})\mathbf{x}_\sigma\\
		&=\sum_{(\mathbf{k},\mathbf{j'})\sim (\mathbf{i},\mathbf{j})}
		\sigma(\mathbf{i},\mathbf{j};\mathbf{k},\mathbf{j'})
		\alpha(\epsilon_{\mathbf{k}}+\epsilon_{\mathbf{j'}},\epsilon_{\mathbf{j'}})
		\gamma(\epsilon_{\mathbf{k}},\sigma)v_{\mathbf{k}\sigma}\otimes v^{\otimes r-l},
	\end{align*}
where $\mathbf{j}'\in I(m|n,l)$ and $(\mathbf{i},\mathbf{j})\in\Omega(m|n,l)$.
	
	Now we compare the coefficients of above two equations
	\begin{align*}
	&\ \ \ \sigma(\mathbf{i},\mathbf{j};\mathbf{k}\cdot \sigma,\mathbf{j'}\cdot \sigma)
	\alpha(\epsilon_{\mathbf{k}\cdot \sigma}+\epsilon_{\mathbf{j'}\cdot \sigma},\epsilon_{\mathbf{j'}\sigma})
	\gamma(\epsilon_{\mathbf{j'}},\sigma)\\
	&=
	\sigma(\mathbf{i},\mathbf{j};\mathbf{k},\mathbf{j'})
	\sigma(\mathbf{k},\mathbf{j'};\mathbf{k}\cdot \sigma,\mathbf{j'}\cdot \sigma)
	\alpha(\epsilon_{\mathbf{k}}+\epsilon_{\mathbf{j'}},\epsilon_{\mathbf{j'}})\\
	&\ \ \ \prod_{s<t,\sigma^{-1}s>\sigma^{-1}t}(-1)^{\bar{j_s'}(\bar{k}_t+\bar{j'_t})+\bar{j'_t}(\bar{k}_s+\bar{j'_s})}
	\gamma(\epsilon_{\mathbf{j'}},\sigma)\\
	&=\sigma(\mathbf{i},\mathbf{j};\mathbf{k},\mathbf{j'})
	\sigma(\mathbf{k},\mathbf{j'};\mathbf{k}\cdot \sigma,\mathbf{j'}\cdot \sigma)
	\alpha(\epsilon_{\mathbf{k}}+\epsilon_{\mathbf{j'}},\epsilon_{\mathbf{j'}})\\
	&\ \ \ \gamma(\epsilon_{\mathbf{k}}+\epsilon_{\mathbf{j'}+\mathbf{j'},\sigma})
	\gamma(\epsilon_{\mathbf{k}}+\epsilon_{\mathbf{j'}},\sigma)
	\gamma(\epsilon_{\mathbf{k}}+\epsilon_{\mathbf{j'}+\mathbf{j'}},\sigma)
	\gamma(\epsilon_{\mathbf{j'},\sigma})\\
	&=\sigma(\mathbf{i},\mathbf{j};\mathbf{k},\mathbf{j'})
	\alpha(\epsilon_{\mathbf{k}}+\epsilon_{\mathbf{j'}},\epsilon_{\mathbf{j'}})
	\gamma(\epsilon_{\mathbf{k}},\sigma).
\end{align*}
	
	So $\xi_{\mathbf{i},\mathbf{j},l}\mathbf{x}_\sigma=\mathbf{x}_\sigma\xi_{\mathbf{i},\mathbf{j},l}$.
	
	(3) Different from (1) and (2), the argument here will be unified, independent of the parity of $v$.
	
	By (2),
	we only need to verify $\Xi$ preserving the relation (\ref{h6}).
	For any elements $w\notin V_{\underline{l}}^{\otimes r}$ and $w\notin V_{\underline{k}}^{\otimes r}$,
	$$w\cdot\mathbf{x}_\delta^{(l)}\mathbf{x}_\gamma^{(k)}=0.$$
	If $v_{\mathbf{j},\underline{l}}\in V_{\underline{l}}^{\otimes r}$, then
	 $$v_{\mathbf{j},\underline{l}}\cdot\mathbf{x}_\delta^{(l)}\mathbf{x}_\gamma^{(k)}=\gamma(\epsilon_{\mathbf{j}},\delta)v_{\mathbf{j}\cdot\delta,\underline{l}}\cdot\mathbf{x}_\gamma^{(k)}=0.$$
	If $v_{\mathbf{j},\underline{k}}\in V_{\underline{k}}^{\otimes r}$, then
	$$v_{\mathbf{j},\underline{k}}\cdot\mathbf{x}_\delta^{(l)}\mathbf{x}_\gamma^{(k)}=0.$$
	The Lemma is proved.
\end{proof}

\section{Levi-type Schur-Sergeev duality}
Let $D(m|n,r):=\Xi(\underline{\mathcal{H}}_{r})$ and $D(m|n,r)_l:=\Xi_l(\underline{\mathcal{H}}_{r}^l)$. Extend the action of $D(m|n,r)_l$ on the whole of $\underline{V}^{\otimes r}$ by setting $D(m|n,r)_l(\underline{V}^{\otimes r}_k)=0$ for $k\neq l$.
So $D(m|n,r)_l$ is a subalgebra of $D(m|n,r)$, and $\bigoplus_{l=0}^rD(m|n,r)_l=\sum_{l=0}^rD(m|n,r)_l\subseteq D(m|n,r)$. Similarly, we extend any given $\varphi\in\mathrm{End}_{S'(m|n,l)}(\underline{V}_l^{\otimes r})$ to an element $\tilde{\varphi}$ of
$\mathrm{End}_{S'(m|n,l)}(\underline{V}^{\otimes r})$ by defining $\tilde{\varphi}(\underline{V}_k^{\otimes r})=0$ for $k\neq l$. Then $\tilde{\varphi}\in\mathrm{End}_{S'(m|n,r)}(\underline{V}^{\otimes r})$.
And we have $\mathrm{End}_{S'(m|n,r)}(\underline{V}^{\otimes r})\supseteq\bigoplus\limits_{l=0}^r\mathrm{End}_{S'(m|n,l)}(\underline{V}_l^{\otimes r})$.

In this section, we will establish the duality between the Schur superalgebras $S'(m|n,r)$ and $D(m|n,r)$.
%\subsection{Levi-type Schur-Sergeev duality for even enhanced vector}

%In this subsection, make $\bar{v}=\bar{0}$.

\begin{proposition}\label{l5}
	The action of $S(m|n,l)$ on $\underline{V}_l^{\otimes r}$ is faithful, i.e., if $S(m|n,l)w_l=0$ for $w_l\in\underline{V}_{l}^{\otimes r}$, then $w_l=0$.
\end{proposition}
\begin{proof}
	The proof is independent of the parity of $v$.
	
	Let $w_l=\sum\limits_{I\subseteq\underline{r},|I|=l}w_I$ with $w_I\in\underline{V}_I^{\otimes r}$.
	Since $S(m|n,l)\underline{V}_{I}^{\otimes r}\subseteq\underline{V}_{I}^{\otimes r}$,
	it follows that $S(m|n,l)w_I=0$.
	Suppose $w_{\underline{l}}=\sum\limits_{\mathbf{t}\in I(m|n,l)}\lambda_{\mathbf{t}}v_{\mathbf{t}}\otimes v^{\otimes r-l}$.
	Then for any $(\mathbf{i},\mathbf{i})\in\Omega(m|n,l)$, we have 
	\begin{align*}
	0=\xi_{\mathbf{i},\mathbf{i},l}w_{\underline{l}}&=\sum_{\mathbf{t}\in I(m\mid n,l)}\lambda_{\mathbf{t}}\xi_{\mathbf{i},\mathbf{i},l}(v_{\mathbf{t}}\otimes v^{\otimes r-l})\\
	&=\sum_{\mathbf{t}\in I(m\mid n,l)}\lambda_{\mathbf{t}}\sum_{(\mathbf{k},\mathbf{t})\sim (\mathbf{i},\mathbf{i})}\sigma(\mathbf{i},\mathbf{i};\mathbf{k},\mathbf{t})\alpha(\epsilon_{\mathbf{k}}+\epsilon_{\mathbf{t}},\epsilon_{\mathbf{t}})v_{\mathbf{k}}\otimes v^{\otimes r-l}\\
	&=\sum_{\mathbf{t}\in I(m\mid n,l)}\lambda_{\mathbf{t}}\sum_{\mathbf{t}\sim\mathbf{i}}v_{\mathbf{t}}\otimes v^{\otimes r-l}.
    \end{align*}

	%For each $\mathbf{t}\in I(m|n,l)$ with $\mathbf{t}\sim\mathbf{j}$, we fix $w_{\mathbf{t},\mathbf{k}}\in\mathfrak{S}_l$ such that $\mathbf{i}w_{\mathbf{t},\mathbf{k}}=\mathbf{k}$ with
	%$\mathbf{j}w_{\mathbf{t},\mathbf{k}}=\mathbf{t}$. So the  coefficient of $v_{\mathbf{k}}\otimes v^{\otimes r-l}$ is $0$. This is
	
    Hence $\lambda_{\mathbf{t}}=0$ for all $\mathbf{t}\sim \mathbf{i}$.
    Furthermore, $\lambda_{\mathbf{t}}=0$ for all $\mathbf{t}\in I(m|n,l)$.
    Then we know $w_{\underline{l}}=0$.
	
	For each $w_I$, we fix $\tau_I\in \mathfrak{S}_r$ such that $\tau_I(w_I)\in\underline{V}_{\underline{l}}^{\otimes r}$.
	Because elements of $S(m|n,l)$ commute with elements of $\mathfrak{S}_r$,
	$0=\tau_IS(m|n,l)w_I=S(m|n,l)\tau_I(w_I)$, i.e., $\tau_I(w_I)=0$. So $w_I=0$ and $w_l=0$.
    The proposition is proved.
\end{proof}

\begin{theorem}\label{l4}
%The algebra of morphism $\rho_r:S'(m|n,r)\rightarrow \mathrm{End}_{D(m|n,r)}(\underline{V}^{\otimes r})$ is an isomorphism,
%where $D(m|n,r)=\Xi(\underline{\mathcal{H}}_{r})\subseteq \mathrm{End}_k(\underline{V}^{\otimes r})$,
%the image of the action of $\underline{\mathcal{H}}_{r}$ on $\underline{V}^{\otimes r}$.
%$\mathrm{End}_{S'(m|n,r)}(\underline{V}^{\otimes r})\cong D(m|n,r)$ for $r\leq m+n$.
Keep the above notations. The following Levi-type Schur-Sergeev duality holds:
\begin{align*}
    S'(m|n,r)\cong&\mathrm{End}_{D(m|n,r)}(\underline{V}^{\otimes r});\\
	&\mathrm{End}_{S'(m|n,r)}(\underline{V}^{\otimes r})\cong D(m|n,r) \text{ for } r\leq m+n.
\end{align*}
\end{theorem}

\begin{proof}
\textbf{Step 1} We prove the first isomorphism in the following.	

According to Lemma \ref{l1}(2), it is clear that $\rho_r(S'(m|n,r))\subset\mathrm{End}_{D(m|n,r)}(\underline{V}^{\otimes r})$.
Recall that the representation $\rho_r:S(m\mid n,r)\longrightarrow\mathrm{End}_{\mathbb{K}\mathfrak{S}_r}(V^{\otimes r})$ is an isomorphism  introduced in Lemma \ref{Lemma5.1} and Theorem \ref{theorem5.2}(1).
It follows that $S(m+1|n,r)\cong\mathrm{End}_{\mathbb{K}\mathfrak{S}_r}(\underline{V}^{\otimes r})$ if $\bar{v}=\bar{0}$ and $S(m|n+1,r)\cong\mathrm{End}_{\mathbb{K}\mathfrak{S}_r}(\underline{V}^{\otimes r})$ if $\bar{v}=\bar{1}$.

Let $\Omega$ denote $\Omega(m+1|n,r)$ if $\bar{v}=\bar{0}$ and denote $\Omega(m|n+1,r)$ if $\bar{v}=\bar{1}$.

Let $\phi\in\mathrm{End}_{D(m|n,r)}(\underline{V}^{\otimes r})
\subset\mathrm{End}_{\mathbb{K}\mathfrak{S}_r}(\underline{V}^{\otimes r})\cong S(m+1|n,r)\text{ or }S(m|n+1)$ (it depends on the parity of $v$),
hence
$$\phi=\sum_{(\mathbf{i},\mathbf{j})\in \Omega}a_{\mathbf{i},\mathbf{i}}
\xi_{\mathbf{i},\mathbf{j}}
=\sum_{l=0}^{r}
\sum_{(\mathbf{i},\mathbf{j})\in \Omega,\atop \text{rk}\mathbf{j}=l}
a_{\mathbf{i},\mathbf{j}}\xi_{\mathbf{i},\mathbf{j}}=\sum_{l=0}^{r}\phi_l,$$
where $\phi_l=\sum\limits_{(\mathbf{i},\mathbf{j})\in \Omega,\atop rk\mathbf{j}=l}
a_{\mathbf{i},\mathbf{j}}\xi_{\mathbf{i},\mathbf{j}}$ and $\text{rk}\mathbf{j}=\left |\{i\mid v'_{\mathbf{j}}=v'_{j_1}\otimes...\otimes v'_{j_r},v'_{j_i}\in basis(V)\}\right |$.

Recall that $\xi_{\mathbf{i},\mathbf{j}}v'_{\mathbf{l}}
=\sum\limits_{(\mathbf{k},\mathbf{l})\sim (\mathbf{i},\mathbf{j})}
\sigma(\mathbf{i},\mathbf{j};\mathbf{k},\mathbf{l})
\alpha(\epsilon_{\mathbf{k}}+\epsilon_{\mathbf{l}},\epsilon_{\mathbf{l}})
v'_{\mathbf{k}}$,
so $\phi|_{\underline{V}_{l}^{\otimes r}}=\phi_l|_{\underline{V}_{l}^{\otimes r}}$
and $\phi_l|_{\underline{V}_{k}^{\otimes r}}=0$ when $k\neq l$.

(a) Show that $\phi({\underline{V}_{l}^{\otimes r}})\subseteq \underline{V}_{l}^{\otimes r}$.
	
	Choose $v'_{\mathbf{j}}\in \underline{V}_l^{\otimes r}$
	and fix $w\in\mathfrak{S}_r$ such that $v'_{\mathbf{j}}\cdot w=\gamma(\epsilon'_{\mathbf{j}},w)v_{\mathbf{j'},l}$.
	
	Suppose $\phi(v'_{\mathbf{j}})=u_l+u_k$ with $u_l\in \underline{V}_{l}^{\otimes r}$ and
	$u_k\notin \underline{V}_{l}^{\otimes r}$.
	Then $\phi\cdot w\cdot x_{id}^{\underline{l}}\cdot w^{-1}(v'_{\mathbf{j}})=\phi(v'_{\mathbf{j}})=u_l+u_k$.
	In another way $w\cdot x_{id}^{\underline{l}}\cdot w^{-1}\cdot\phi(v'_{\mathbf{j}})=w\cdot x_{id}^{\underline{l}}\cdot w^{-1}(u_l+u_k)\in \underline{V}_{l}^{\otimes r}$. So $u_k=0$.
	Thus $\phi({\underline{V}_{l}^{\otimes r}})\subseteq \underline{V}_{l}^{\otimes r}$.
	
(b) As the above argument, $\phi_l({\underline{V}_{\underline{l}}^{\otimes r}})\subseteq \underline{V}_{\underline{l}}^{\otimes r}$.

(c) From (b), we can make $\phi_l|_{\underline{V}_{\underline{l}}^{\otimes r}}=
	(\phi'_l\otimes id^{\otimes r-l})|_{\underline{V}_{\underline{l}}^{\otimes r}}$, where $\phi'_l\in \mathrm{End}_{\mathbb{K}}(V^{\otimes l})$.
	
	From $\phi\cdot x_\sigma^{\underline{l}}=x_\sigma^{\underline{l}}\cdot\phi$,
	it can conclude that $(\sigma\cdot\phi'_l)\otimes id^{\otimes r-l}=
	(\phi'_l\cdot \sigma)\otimes id^{\otimes r-l}$ for any $\sigma\in\mathfrak{S}_l$,
	then $\phi'_l\in \mathrm{End}_{\mathbb{K}\mathfrak{S}_l}(V^{\otimes l})\cong S(m|n,l)$.
	Recall the  embedding $S(m|n,l)\hookrightarrow S'(m|n,r)$ mentioned in 2.5.2, then we can find $\xi\in S'(m|n,r)$ such that $\xi|_{\underline{V}_{\underline{l}}^{\otimes r}}=(\phi'_l\otimes id^{\otimes r-l})|_{\underline{V}_{\underline{l}}^{\otimes r}}=\phi_l|_{\underline{V}_{\underline{l}}^{\otimes r}}$
	and $\xi|_{\underline{V}_{\underline{k}}^{\otimes r}}=0$ for $k\neq l$.
	
	So $\phi_l|_{\underline{V}_{I}^{\otimes r}}
	=w^{-1}\cdot \phi_l\cdot w|_{\underline{V}_{I}^{\otimes r}}
	=w^{-1}\cdot\phi_l|_{\underline{V}_{\underline{l}}^{\otimes r}}=w^{-1}\cdot\xi|_{\underline{V}_{\underline{l}}^{\otimes r}}
	=\xi|_{\underline{V}_{I}^{\otimes r}}$,
	where $w\in\mathfrak{S}_r$ such that $\underline{V}_{I}^{\otimes r}w=\underline{V}_{\underline{l}}^{\otimes r}$.
	Then $\phi_l=\xi\in S'(m|n,r)$ and $\phi\in S'(m|n,r)$.
	
	From (a), (b) and (c), $\rho_r(S'(m|n,r))=\mathrm{End}_{D(m|n,r)}(\underline{V}^{\otimes r})$. With the Lemma \ref{LLe}, we obtain $S'(m|n,r)\cong\mathrm{End}_{D(m|n,r)}(\underline{V}^{\otimes r})$.

\textbf{Step 2}	We prove the second isomorphism in the following.
	Recall that in 2.5.2, we have $S'(m|n,r)\cong\bigoplus\limits_{l=0}^{r}S(m|n,l)$.
	In particular, $S(m|n,k)(\underline{V}_l^{\otimes r})=0$ for any $l\neq k$.
	For any $\psi\in\mathrm{End}_{S'(m|n,r)}(\underline{V}^{\otimes r})$,
	any $\xi_k\in S(m|n,k)$, and $w\in \underline{V}^{\otimes r}_l$, suppose $\psi(w)=\sum\limits_{k=0}^rw_k$ with $w_k\in \underline{V}^{\otimes r}_k$, then
	$$\psi\sum\limits_{k=0}^r\xi_k(w)=\psi\xi_l(w)=\xi_l\psi(w)=\xi_l(w_l)\in\underline{V}^{\otimes r}_l$$
	and
	$$\sum\limits_{k=0}^r\xi_k\psi(w)=\sum\limits_{k=0}^r\xi_k(w_k).$$
	Then $\sum\limits_{k=0}^r\xi_k(w_k)\in\underline{V}^{\otimes r}_l$.
	So $\xi_kw_k=0$ for $k\neq l$. This means $S(m|n,k)w_k=0$. By Proposition \ref{l5}, $w_k=0$ for all $k\neq l$.
	Thus $\psi(\underline{V}^{\otimes r}_l)\subseteq\underline{V}^{\otimes r}_l$.

	(a) Prove $\mathrm{End}_{S'(m|n,r)}(\underline{V}^{\otimes r})=\bigoplus\limits_{l=0}^r\mathrm{End}_{S'(m|n,l)}(\underline{V}_l^{\otimes r})$.

From the above we know for any given $\psi\in\mathrm{End}_{S'(m\mid n,r)}(\underline{V}^{\otimes r})$,
$\psi(\underline{V}^{\otimes r}_l)\subseteq\underline{V}^{\otimes r}_l$ for $0\leq l\leq r$.
Let $\psi_l=\psi|_{\underline{V}^{\otimes r}_l}$ for $0\leq l\leq r$. Then $\psi_l\in\mathrm{End}_{S'(m|n,r)}(\underline{V}^{\otimes r}_l)$ and $\psi=\sum_{l=0}^{r}\psi_l$. So $\mathrm{End}_{S'(m\mid n,r)}(\underline{V}^{\otimes r})\subseteq\bigoplus\limits_{l=0}^r\mathrm{End}_{S'(m|n,l)}(\underline{V}_l^{\otimes r})$.
With the beginning of this section, we have $\mathrm{End}_{S'(m|n,r)}(\underline{V}^{\otimes r})=\bigoplus\limits_{l=0}^r\mathrm{End}_{S'(m|n,l)}(\underline{V}_l^{\otimes r})$.

(b) Prove $\mathrm{End}_{S'(m|n,l)}({\underline{V}_{\underline{l}}^{\otimes r}})\subseteq D(m|n,r)_l=\Xi_l(\mathcal{H}_r^l)$.

For any $\varphi\in\mathrm{End}_{S'(m|n,l)}({\underline{V}_{\underline{l}}^{\otimes r}})$,
it can be written as $\varphi'\otimes id^{\otimes r-l}$ with $\varphi'\in \mathrm{End}_{\mathbb{K}}(V^{\otimes l})$.
By definition,
$$\xi_{\mathbf{i},\mathbf{j},l}(v_{\mathbf{t}}\otimes v^{\otimes r-l})
=\xi_{\mathbf{i},\mathbf{j}}(v_{\mathbf{t}})\otimes v^{\otimes r-l},$$
$$\xi_{\mathbf{i},\mathbf{j},k}(v_{\mathbf{t}}\otimes v^{\otimes r-l})
=0\text{ for }k\neq l,$$
so $\varphi$ commutes with $\xi_{\mathbf{i},\mathbf{j},l}$ if and only if $\varphi'$ commutes with $\xi_{\mathbf{i},\mathbf{j}}$. Because $l\leq r\leq m+n$, by Lemma \ref{theorem5.2}(2), then $\varphi'\in  \mathbb{K}\mathfrak{S}_l\cong\mathrm{End}_{S(m|n,l)}(V^{\otimes l})$.

Thus the homomorphism of algebras $\mathrm{End}_{S'(m\mid n,l)}({\underline{V}_{\underline{l}}^{\otimes r}})\longrightarrow D(m|n,r)_l$ defined by $\varphi\longmapsto \sum\limits_{\sigma\in\mathfrak{S}_l}g_{\sigma}\Xi_l(\mathbf{x}_\sigma)$ is injective, where $\varphi'=\sum\limits_{\sigma\in\mathfrak{S}_l}g_{\sigma}\sigma$ and $g_{\sigma}\in \mathbb{K}$ for all $\sigma\in\mathfrak{S}_l$.

(c) Prove $\mathrm{End}_{S'(m|n,r)}(\underline{V}_l^{\otimes r})=D(m|n,r)_l$ for any $l=0,1,...,r$.

It is obviously that $D(m|n,r)_l\subseteq\mathrm{End}_{S'(m\mid n,l)}(\underline{V}_l^{\otimes r})$ by Lemma \ref{l1}(2).

For any given $\varphi\in\mathrm{End}_{S'(m|n,r)}(\underline{V}_l^{\otimes r})=\bigoplus\limits_{I\subseteq\underline{r},J\subseteq\underline{r},\atop|I|=|J|=l}\mathrm{Hom}_{S'(m\mid n,r)}(\underline{V}_I^{\otimes r},\underline{V}_J^{\otimes r})$,
let $\varphi=\sum\limits_{I\subseteq\underline{r},J\subseteq\underline{r},\atop|I|=|J|=l}\varphi_{I,J}$ where $\varphi_{I,J}\in \mathrm{Hom}_{S'(m|n,r)}(\underline{V}_I^{\otimes r},\underline{V}_J^{\otimes r})$.
Fix $\{\tau_{I,J}\in\mathfrak{S}_r\}$ such that $\Xi_l(\tau_{I,J})(\underline{V}_{I}^{\otimes r})=\underline{V}_J^{\otimes r}$.
Then $\Xi_l(\tau^{-1}_{J,\underline{l}})\cdot\varphi_{I,J}\cdot\Xi_l(\tau_{\underline{l},I})\in\mathrm{End}_{S'(m|n,l)}({\underline{V}_{\underline{l}}^{\otimes r}})\subseteq D(m|n,r)_l$.
So $\varphi_{I,J}\in\Xi_l(\tau_{J,\underline{l}})\cdot D(m|n,r)_l\cdot\Xi_l(\tau_{J,\underline{l}})^{-1}\subset D(m|n,r)_l$
and $\varphi\in D(m|n,r)_l$.

So $\mathrm{End}_{S'(m|n,r)}(\underline{V}_l^{\otimes r})=D(m|n,r)_l$.

(d) Prove $D(m|n,r)=\bigoplus\limits_{l=0}^rD(m|n,r)_l$.

$D(m|n,r)$ is generated by $\{\Xi(\mathbf{x}_\sigma^{(l)}),\sigma\in\mathfrak{S}_l,1\leq l\leq r\}$ and $\{\Xi(\mathbf{s}_i),1\leq i\leq r-1\}$.
And $$\Xi(\mathbf{x}_\sigma^{(l)})=\Xi_l(\mathbf{x}_\sigma^{(l)})\in D(m|n,r)_l,$$
$$\Xi(\mathbf{s}_i)=\sum_{l=0}^{r}\Xi_l(\mathbf{s}_i)\in \bigoplus_{l=0}^rD(m|n,r)_l.$$
So $D(m|n,r)\subseteq\bigoplus\limits_{l=0}^rD(m|n,r)_l$. Therefore $D(m|n,r)=\bigoplus\limits_{l=0}^rD(m|n,r)_l$.

Consequently, $\mathrm{End}_{S'(m|n,r)}(\underline{V}^{\otimes r})=\bigoplus\limits_{l=0}^r\mathrm{End}_{S'(m|n,r)}(\underline{V}_l^{\otimes r})=\bigoplus\limits_{l=0}^rD(m|n,r)_l=D(m|n,r)$.

\end{proof}

%\noindent\textbf{The Proof of Theorem \ref{theorem 1.3}}
%\quad

%\qed
%
%%\iffalse
%\section*{Acknowledgements}
%%%This work is supported by the National Science Foundation
%%%of China (No. 11171254).
%We are grateful to the anonymous referee for his or her careful
%reading of the original manuscript of this paper and giving us some
%invaluable comments.
%\fi

\begin{remark}
	The parabolic  Schur-Sergeev duality for general linear supergroups has been studied in \cite[Chapter 7]{L2022}.
\end{remark}

\bibliographystyle{amsplain}

\end{document}